\documentclass[11pt,a4paper]{article}
\usepackage{amsmath}
\usepackage{amsfonts}
\usepackage{amssymb}
\usepackage{graphicx}
 \usepackage{amsthm}

\usepackage{subcaption}

\usepackage{url}

\usepackage{hyperref}
\usepackage{fancyhdr}
\usepackage[authoryear, round]{natbib}

\usepackage{color}
\def\vint_#1{\mathchoice%
      {\mathop{\kern 0.2em\vrule width 0.6em height 0.69678ex depth -0.58065ex
              \kern -0.8em \intop}\nolimits_{\kern -0.4em#1}}%
      {\mathop{\kern 0.1em\vrule width 0.5em height 0.69678ex depth -0.60387ex
              \kern -0.6em \intop}\nolimits_{#1}}%
      {\mathop{\kern 0.1em\vrule width 0.5em height 0.69678ex depth -0.60387ex
              \kern -0.6em \intop}\nolimits_{#1}}%
      {\mathop{\kern 0.1em\vrule width 0.5em height 0.69678ex depth -0.60387ex
              \kern -0.6em \intop}\nolimits_{#1}}}

\def\vintslides_#1{\mathchoice%
      {\mathop{\kern 0.1em\vrule width 0.5em height 0.697ex depth -0.581ex
              \kern -0.6em \intop}\nolimits_{\kern -0.4em#1}}%
      {\mathop{\kern 0.1em\vrule width 0.3em height 0.697ex depth -0.604ex
              \kern -0.4em \intop}\nolimits_{#1}}%
      {\mathop{\kern 0.1em\vrule width 0.3em height 0.697ex depth -0.604ex
              \kern -0.4em \intop}\nolimits_{#1}}%
      {\mathop{\kern 0.1em\vrule width 0.3em height 0.697ex depth -0.604ex
              \kern -0.4em \intop}\nolimits_{#1}}}

\newcommand{\aveint}[2]{\mathchoice%
      {\mathop{\kern 0.2em\vrule width 0.6em height 0.69678ex depth -0.58065ex
              \kern -0.8em \intop}\nolimits_{\kern -0.45em#1}^{#2}}%
      {\mathop{\kern 0.1em\vrule width 0.5em height 0.69678ex depth -0.60387ex
              \kern -0.6em \intop}\nolimits_{#1}^{#2}}%
      {\mathop{\kern 0.1em\vrule width 0.5em height 0.69678ex depth -0.60387ex
              \kern -0.6em \intop}\nolimits_{#1}^{#2}}%
      {\mathop{\kern 0.1em\vrule width 0.5em height 0.69678ex depth -0.60387ex
              \kern -0.6em \intop}\nolimits_{#1}^{#2}}}

\begin{document}

\newtheorem{theorem}{Theorem}[section]
\newtheorem{proposition}[theorem]{Proposition}
\newtheorem{lemma}[theorem]{Lemma}
\newtheorem{eg}[theorem]{Example}
\newtheorem{definition}[theorem]{Definition}
\newtheorem{remark}[theorem]{Remark}
\newtheorem{notn}[theorem]{Notation}
\newtheorem{corollary}[theorem]{Corollary}
\newtheorem{conjecture}[theorem]{Conjecture}
\newtheorem{assumption}[theorem]{Assumption}
\newtheorem{condn}[theorem]{Condition}
\newtheorem{example}[theorem]{Example}

\newcommand\numberthis{\addtocounter{equation}{1}\tag{\theequation}}

\newcommand{\IP}{\mathbb P}
\newcommand{\IQ}{\mathbb Q}
\newcommand{\IE}{\mathbb E}
\newcommand{\IR}{\mathbb R}
\newcommand{\IZ}{\mathbb Z}
\newcommand{\IN}{\mathbb N}
\newcommand{\IT}{\mathbb T}
\newcommand{\IC}{\mathbb C}

\newcommand{\rmP}{\mathrm{P}}
\newcommand{\rmE}{\mathrm{E}}
\newcommand{\rmT}{\mathrm{T}}
\newcommand{\rmt}{\mathrm{t}}

\newcommand{\tX}{\tilde{X}}

\newcommand{\hX}{\hat{X}}
\newcommand{\hS}{\hat{S}}
\newcommand{\hq}{\hat{q}}
\newcommand{\hQ}{\hat{Q}}

\newcommand{\cP}{\mathcal{P}}
\newcommand{\cO}{\mathcal{O}}
\newcommand{\cA}{\mathcal{A}}
\newcommand{\cL}{\mathcal{L}}
\newcommand{\cG}{\mathcal{G}}
\newcommand{\cC}{\mathcal{C}}
\newcommand{\cV}{\mathcal{V}}
\newcommand{\cE}{\mathcal{E}}
\newcommand{\cT}{\mathcal{T}}
\newcommand{\cD}{\mathcal{D}}
\newcommand{\cZ}{\mathcal{Z}}
\newcommand{\cB}{\mathcal{B}}
\newcommand{\cK}{\mathcal{K}}
\newcommand{\cN}{\mathcal{N}}
\newcommand{\cX}{\mathcal{X}}
\newcommand{\cF}{\mathcal{F}}
\newcommand{\dd}{\mathrm{d}}
\newcommand{\ind}{\mathbf{1}}

\newcommand{\vs}{\mathbf{s}}

\pagestyle{fancy}
\lhead{}
\chead{Evolution in a fluctuating environment}
\rhead{}

\numberwithin{equation}{section}

\title{\large{\bf The spatial Lambda-Fleming-Viot
process with fluctuating selection}}
                                           
\author{ \begin{small}
\begin{tabular}{lll}                              
Niloy Biswas & Alison M.~Etheridge &
Aleksander Klimek \thanks{Supported by EPSRC grant
number EP/L015811/1}
\\   
Department of Statistics
&Department of Statistics
&
Max Planck Institute\\       
Harvard University & Oxford University & for Mathematics in Natural Sciences\\                   
Science Center 400 Suite& 24--29 St Giles & Inselstrasse 22\\                                                         
One Oxford Street & Oxford OX1 3LB & 04103 Leipzig \\
Cambridge, MA 02138-2901 & UK & DE\\
USA &  &  \\  niloy\_biswas@g.harvard.edu                      
 & etheridg@stats.ox.ac.uk &  klimek@mis.mpg.de    \\
\end{tabular}
\end{small}}


\maketitle

\begin{abstract}
We are interested in populations in which the fitness of different 
genetic types fluctuates in time and space, driven by temporal and 
spatial fluctuations in the environment. For
simplicity, our population is assumed to be composed of just two genetic types.
Short bursts of selection acting in opposing directions drive to
maintain both types at intermediate frequencies, while the fluctuations
due to `genetic drift' work to eliminate variation in the population.

We consider first a population with no spatial structure, modelled by an
adaptation of the
Lambda (or generalised) Fleming-Viot process, and derive a stochastic 
differential equation as a scaling limit. This amounts to a limit
result for a Lambda-Fleming-Viot process in a rapidly fluctuating random 
environment. We then extend to a population 
that is distributed across a spatial continuum, which we model through
a modification of the spatial Lambda-Fleming-Viot process with selection.
In this setting we show that the scaling limit is a stochastic partial
differential equation. As is usual with spatially distributed populations,
in dimensions greater than one, the `genetic drift' disappears in the 
scaling limit, but here we retain some stochasticity due to the 
fluctuations in the environment, resulting in a stochastic p.d.e.~driven
by a noise that is white in time but coloured in space. 

We discuss the (rather limited) 
situations under which there is a duality with a system of 
branching and annihilating particles.
We also write down a system of equations that captures the frequency of 
descendants of particular subsets of the population and use this same
idea of `tracers', which we learned from \cite{hallatschek/nelson:2008}
and \cite{durrett/fan:2016}, in numerical experiments with a closely related
model based on the classical Moran model. 
\vspace{.1in}


\vspace{.1in}

\noindent {\bf MSC 20}10 {\bf Subject Classification:}  Primary: 
60G57, 
60J25, 
92D15  
\\
Secondary:  
60J75, 
60G55 

\end{abstract}



\maketitle
\tableofcontents

\bibliographystyle{plainnat}
\DeclareRobustCommand{\VAN}[3]{#3}

\section{Introduction}

A fundamental challenge in population genetics is to understand the 
balance between
adaptive processes ({\em selection}) and random neutral processes 
({\em genetic drift}). 
The most studied example of adaptation 
is directional selection acting on a single
genetic locus. In the simplest model, each individual is either of type
$a$ or type $A$ at the locus under selection, and the relative fitnesses of
individuals carrying the two types is $1+s_0:1$, for some small parameter $s_0$.
At least provided that random fluctuations don't eliminate the favoured 
type before it can become established, 
natural
selection will act to remove variability from the population until,
in the absence of mutation, 
everyone is of the favoured type. However, there are other forms of 
selection that act to maintain genetic variation. 
In this paper we are concerned with
populations that are subject to 
changing environmental conditions, that cause relative fitnesses of different 
genotypes to fluctuate in time and space.
To quote \cite{gillespie:2004}, {\em 
``If fitnesses do depend on the state of
the environment, as they surely must, then they must just as 
assuredly change in both time and
space, driven by temporal and spatial fluctuations in the environment.''}

We shall suppose that our population occurs in just two types ({\em alleles}), 
$\{a,A\}$ and that
the environment fluctuates between two states, in the first of which $a$,
and in the second of which $A$, is favoured. 
We suppose that selection is sufficiently strong that if the environment 
did not
fluctuate, the favoured type would rapidly fix in the 
population, but that there is a 
`balance' between the two environments so that both types can be maintained
at non-trivial frequencies for long periods of time.
If the population has no spatial structure, then over large timescales
the frequency of $a$-alleles
can be modelled by a stochastic differential equation:
\begin{equation}
\label{basic sde}
\dd p={\vs}^2p(1-p)(1-2p)\dd t +\sqrt{p(1-p)}\mathrm{d}B^1+\sqrt{2}{\vs}p(1-p)\mathrm{d}B^2,
\end{equation}
where $B^1$ and $B^2$ are independent
Brownian motions, the first (as we shall explain in Section~\ref{nospace})
capturing the randomness due to genetic drift (that is the
randomness due to reproduction in a finite population), 
the second encoding the random fluctuations in the environment
(which are assumed to happen quickly on evolutionary timescales). The constant
${\vs}$ is a scaled selection coefficient (see Section~\ref{nospace}).
For a population distributed across a one-dimensional spatial continuum,
one can write down an analogous stochastic partial
differential equation:
\begin{multline}
\label{spde in one dimension}
\dd w(t,x)= \frac{1}{2}\Delta w(t,x)\dd t +{\vs}^2w(t,x)(1-w(t,x)(1-2w(t,x))\dd t
\\
+\sqrt{w(t,x)(1-w(t,x))}{\mathcal W}(\dd t,\dd x)
+\sqrt{2}{\vs}w(t,x)(1-w(t,x))W(\dd t, \dd x),
\end{multline}
where ${\mathcal W}$ is space-time white noise (capturing genetic drift)
and the independent noise $W$ is white
in time, but may be coloured in space reflecting spatial correlations in 
the environmental fluctuations. In the biologically most relevant case
of two dimensions, this equation has no solution, and we must find a 
different approach.

The difficulties with modelling genetic drift in
populations evolving in higher dimensional spatial
continua, often referred to as `the pain in the torus', are well known; see
\cite{barton/etheridge/veber:2013} for a review. They can be overcome 
using the spatial Lambda-Fleming-Viot process, introduced in  \cite{etheridge:2008} 
and rigorously constructed in, \cite{barton/etheridge/veber:2010},
and here we adapt that model to incorporate fluctuating selection. 

Our first result deals with the non-spatial case. We take a scaling limit
of the Lambda-Fleming-Viot process and recover~(\ref{basic sde}),
which coincides with that obtained by \cite{gillespie:2004} as a scaling
limit of a Wright-Fisher type model. We then turn to the scaling limit
of the spatial Lambda-Fleming-Viot process with fluctuating selection.
In dimension one, the limiting process coincides (up to constants)
with the stochastic 
p.d.e.~(\ref{spde in one dimension}).
In higher dimensions, the term corresponding to genetic drift
vanishes in the limit, but the effects of the fluctuations in the environment
can still persist, resulting in a stochastic p.d.e.~driven by (spatially) coloured
noise. 


Our ultimate aim is to find ways to distinguish the effects of spatial and 
temporal environmental fluctuations on genetic data. This would involve understanding
the genealogical trees relating individuals in a sample from the population. 
Although for our prelimiting model we can write down an analogue of the 
ancestral selection graph of \cite{krone/neuhauser:1997},   \cite{neuhauser/krone:1997}, 
which tracks all `potential ancestors' of individuals in a sample from the
population, this process seems to be rather unwieldy. Moreover, when we apply our 
rescaling, the scaled ancestral selection graphs do not converge and we have not
found a satisfactory way to extract genealogies for the limiting model.
When selection does not fluctuate, the original
ancestral selection graph can be thought of as a moment dual to the forwards in
time diffusion describing allele frequencies in the population.
It is natural to ask whether there are other dual processes that we 
could exploit when selection fluctuates.
Our attempts to find a useful dual for the equation~(\ref{spde in one dimension})
have met with limited success, but in
Section~\ref{duality} we show that (after an affine transformation)
there are circumstances in which a branching and annihilating dual exists.


In the absence of a useful dual process,
instead we take a first step towards understanding ancestry in
the population by following an 
interesting approach of \cite{hallatschek/nelson:2008} and, more recently,
 \cite{durrett/fan:2016} which uses the idea of `tracers' to explore the way in 
which descendants of a subpopulation of the type $a$ individuals evolve forwards in time.
In Section~\ref{tracers} we write down the system of stochastic p.d.e.'s that will
determine the tracer dynamics. This idea is exploited further in our 
numerical experiments
of Section~\ref{numerics}.

The rest of the article is laid out as follows. In Section~\ref{biology}
we very briefly outline some of the biological background. 
In Section~\ref{nospace} we consider the case in which the population
has no spatial structure. To prepare the ground for the case of spatially
structured populations, we work with the Lambda-Fleming-Viot process
(also sometimes known as the generalised Fleming-Viot process) that was
introduced in \cite{donnelly/kurtz:1999}, \cite{bertoin/legall:2003}.
In particular, we investigate different scaling limits, reflecting
longtime behaviour of the process for different balances between the 
rate of changes of environment and the strength of selection.
In Section~\ref{results}
we define the spatial Lambda-Fleming-Viot process with fluctuating selection and
give a precise statement of our scaling result
for this model. 
In Section~\ref{duality}, we discuss the situations in which we can 
investigate the limiting process through duality with a system of 
branching and annihilating particles. 
Tracers are introduced in Section~\ref{tracers}
and then explored numerically (for a Moran model of a subdivided population)
in Section~\ref{numerics}. The proof of
our main scaling limit is in Section~\ref{scaling proof}.
The appendices contain some (important) technical results that we
require in the course of the proofs.

\paragraph{Acknowledgement }We should like to thank Tom Kurtz and Amandine 
V\'eber for extremely helpful discussions
and two anonymous referees for a careful reading of the manuscript and valuable suggestions.

\section{Biological background}
\label{biology}

In this section we outline the biological context for this work. 
Although not a prerequisite for understanding the mathematics of 
subsequent sections, it explains our motivation for tackling this 
particular scaling limit.

Suppose that a gene occurs in just two forms that, because of environmental
fluctuations, each finds
itself subject to short alternating bursts 
of positive and
negative selection. Even if these changes are happening on a much faster 
scale than neutral evolution,
they may influence gene frequencies. For example, 
in diploid individuals 
(carrying two copies of the gene), a heterozygote
(carrying one allele of each type) may have higher mean fitness, when we take account of
different environments, than either homozygote, and so allelic variation
can be maintained for long periods, even though at any given time the population is 
subject to directional selection. 
This {\em marginal overdominance} is an example of {\em balancing selection}. 
In equation~(\ref{basic sde}) we see this in the deterministic term
($sp(1-p)(1-2p)\mathrm{d}t$) on the right hand side.  
1969, \cite{wright:1969} observed that spatial heterogeneity in the direction of selection 
combined with density dependent reproduction 
can also lead to {\em balanced polymorphism}, that is adaptive alleles are held at intermediate frequencies for long periods (see also \cite{delph/kelly:2014} and 
references therein).

One of the recurring arguments in evolutionary biology is whether 
evolution occurs principally through natural selection or through 
neutral processes, in which no particular genetic type is
favoured, such as genetic drift. 
A data set that has sat at the heart of this 
debate for the last 70 years is a time series of changes in the 
genotype frequency of a polymorphism of the Scarlet Tiger Moth,
{\em Callimporpha (Panaxia) dominula}, in an isolated population
at Cothill Fen near Oxford, UK.  
\cite{fisher/ford:1947} found that the proportion of a certain {\em medionigra} 
allele in the population increased significantly between
1929 and 1941 from $1.2$\% to $11.1$\%, and decreased to $5.2$\% 
between 1941 and 1946. 
They concluded,
``$\ldots$ the observed fluctuations in generations are much greater than 
could be ascribed to random survival only. Fluctuations in natural selection 
must therefore be responsible for them.''.
Fisher (a strong proponent of the importance of 
selection) was challenged by Wright (a champion of
genetic drift) who argued that multiple factors could be at play and,
moreover, Fisher may have underestimated the strength of genetic drift.
 (2005, \cite{ohara:2005} analysed the, by then 60 year long, time series of data from
Cothill and concluded that most of the pattern of 
variation in the population should
be attributed to genetic drift. 
Moreover, although selection is acting, mean fitness
barely increased. 

It is unusual to have such a long time series of data, especially in conjunction
with information about the environment. In general it will also be
far from clear
which genes are undergoing selection, rather one tries to infer the 
action of selection through studying neutral diversity. For most populations, 
it may be very difficult to distinguish fluctuating selection from 
genetic drift. To see why, we recall a model due to Gillespie
that captures the effect of a series of `selective sweeps' through a 
population. 
Suppose that a 
selectively favoured mutation arises at some point on the genome and 
rapidly increases in frequency (until the whole population carries it).
Because genes are arranged on chromosomes, different genes
do not evolve independently of one another. As a result of a process called
recombination, correlations between genes decrease as a function of the
distance between them on the chromosome. Nonetheless, a neutral allele
fortunate enough to be on the same chromosome as the selectively favoured
mutation will itself receive a boost in
its frequency (even if as a result of recombination it doesn't exhaust the
whole population). This boost to
the type at the neutral locus is known as `genetic hitchhiking', 
a term introduced by  \cite{maynardsmith/Haigh:1974}.
Of course, correspondingly, a neutral allele associated with an 
unfavoured type will decrease in frequency.
 \cite{gillespie:2000},  \cite{gillespie:2001}
investigated a model in which strongly selected mutations
which give rise to hitchhiking events occur at the points of a Poisson process. He assumes
that selection is strong enough that the duration of the sweeps causing the hitchhiking
events that affect a given locus is small compared to the time between them so that we can
ignore the possibility that a locus will be subject to two simultaneous hitchhiking events.
He establishes that the
first two moments of the change in allele frequency at the neutral 
locus over the course
of a hitchhiking event take exactly the same form as if they had been
produced by genetic drift over a single generation of reproduction. It is
not hard to see that we will see the same hitchhiking effects under `partial
sweeps' driven by environmental fluctuations. This
process of `genetic draft' induced by selection, 
strongly resembles genetic drift and it may be hard to distinguish the two.
 \cite{barton:2000} also considers the `genetic drift' induced by
hitchhiking.

Not surprisingly, the impact of environmental fluctuations on genetic variation has 
been extensively studied. Nonetheless, 
even in the absence of spatial structure, it remains an open question to 
characterise situations under which fluctuating environmental conditions 
can maintain genetic variation; see e.g.  \cite{novak/barton:2017}.
Moreover, the effects of genetic drift 
have been largely ignored.
This is perhaps because acting in isolation, genetic drift typically impacts gene
frequencies over periods of (tens of) thousands of generations, much longer than the time scales of
climatic fluctuation. 
However, once a particular genotype becomes rare, perhaps as a result of a run of 
unfavourable environments, stochastic fluctuations will be dominated by genetic drift, through
which the genotype can be lost. 

A further challenge in identifying genes that are subject to fluctuating selection is that, even if we can disentangle the effects of drift, 
numerous selection schemes lead to forms of balancing selection.
For example,
in the absence of spatial structure, allele frequency dynamics under fluctuating selection are identical
to those under within-generation fecundity variance polymorphism. In this setting,
 \cite{taylor:2013} shows that the effects on the genealogy at a linked neutral locus will differ.
\cite{fijarczyk/babik:2015} and the references therein provide an 
overview of theoretical and empirical evidence for various forms of
balancing selection and methods for their detection.

Recently, \cite{berglandetal:2014} reported hundreds of polymorphisms in 
{\em Drosophila melanogaster} whose frequencies oscillate among seasons and they attribute this to
strong, temporally variable selection. They also cite evidence that genetic 
(and phenotypic) variation is maintained by temporally fluctuating selection for a variety of other 
organisms.\cite{gompert:2016} proposes an approach to quantifying variable 
selection in populations experiencing both spatial and temporal variations in
selection pressure.
In spite of this combination of theoretical and empirical evidence for the importance of fluctuating
selection, we have only a limited understanding of some basic questions: how many loci 
are subject to temporally fluctuating selection? How strong is that selection?  What is the relationship between temporally and spatially varying selection?

Since natural environments are never truly constant, it is clearly important to 
understand the implication of temporally and spatially varying selection pressures. 
\cite{cvijovicetal:2015} examines some of the implications of temporal
fluctuations. Our work here is a step towards 
a tractable framework in which to
consider the combined effects of spatial and temporal fluctuations.

\section{The non-spatial case}
\label{nospace}

\subsection{The (non-spatial) model}

We first consider a population without
spatial structure. Although we would obtain exactly the same 
scaling limits if we were to use the classical Moran or
Wright-Fisher models as the basis of our approach, c.f.  \cite{gillespie:2004}, 
for consistency with what follows, we shall work with the
(non-spatial) Lambda-Fleming-Viot process.
The key ideas that will be 
required in the spatial setting already appear here, where they
are not obscured by notational complexity.
An analogous scaling limit is obtained for the 
Wright-Fisher model with fluctuating selection (using similar 
reasoning) in \cite{hutzenthaler/pfaffelhuber/printz:2018}.

We shall restrict ourselves to the special case of the 
Lambda-Fleming-Viot process in which reproduction events fall 
at a finite rate, determined by a Poisson process. We shall also
suppose that there are just two types of individual, $\{a,A\}$. In each event,
a parent is chosen from the population immediately before the event, and
a portion $u$ of the population is replaced by offspring of the same type
as the parent. In general the quantity $u$, which we shall call the
{\em impact} of the event, may be random. Selection (on fecundity)
can be incorporated
by weighting the choice of parent, to favour one type or the other, and
we shall extend previous versions of the model to allow 
the direction of selection to fluctuate.
More precisely, we have the following definition.

\begin{definition} [Lambda-Fleming-Viot process with fluctuating selection]
\label{LFVFS}

The Lambda-Fleming-Viot process with fluctuating selection, 
$\{p(t)\}_{t\geq 0}$ is a c\`adl\`ag process taking its values
in $[0,1]$, with $p(t)$ to be interpreted as the proportion of type $a$
individuals in the population at time $t$.

Let $\Pi$ be a Poisson process defined on $\mathbb{R}\times (0,1)\times (0,1)$
with intensity measure 
$\mathrm{d}t\otimes\nu(\mathrm{d}u)\otimes\sigma(\mathrm{d}{\vs})$, 
where $\nu$ and $\sigma$ are some probability measures. 
Moreover, let $\Pi^{env}$ be a rate $\tau^{env}$ Poisson process, 
independent of $\Pi$ (where $\tau^{env}\in(0,\infty))$.
The state of the environment is a random variable $\zeta(t)\in\{-1,1\}$. At the 
times of the Poisson process $\Pi^{env}$, $\zeta$ is resampled uniformly from 
$\{-1,1\}$.

The dynamics of $\{p(t)\}_{t\geq 0}$ can be  
described as follows.
If $(t,u,{\vs})\in \Pi$, a reproduction event occurs. Then:
\begin{enumerate}
\item select a parental type $\kappa \in \{a,A\}$ according to
\begin{align*}
\mathbb{P}[\kappa = a] = \frac{(1+{\vs})p(t_{-})}{1+{\vs}p(t_{-})} \quad \text{if} \quad 
\zeta =-1, \\
\mathbb{P}[\kappa = a] = \frac{p(t_{-})}{1+{\vs}(1-p(t_{-}))} \quad \text{if} \quad 
\zeta =1.
\end{align*}
\item A proportion $u$ of the population immediately before the event dies
and is replaced by offspring of the chosen type, that is
\begin{align*}
p(t) = (1-u)p(t_{-})  + \ind_{\{\kappa = a\}}u.
\end{align*}
\end{enumerate}
\end{definition}

\begin{remark}
If, instead of resampling the environment according to an 
independent Poisson process, we resampled it at each 
reproduction event, by choosing $\nu$ to be 
distributed as the proportion of hitchhikers when a selective sweep
occurs at a random distance from our chosen locus, 
we would recover Gillespie's model of 
genetic draft at a neutral locus linked to loci undergoing 
a sequence of selective sweeps. 

We shall see that the rate of resampling of the environment (relative
to the strength of selection in each event) plays a key role in the long
term behaviour of the population. 

\end{remark}

\subsection{Scaling limits}

In order to simplify the notation still further, we specialise to the
case in which the Poisson point process $\Pi$ of 
Definition~\ref{LFVFS}
has intensity $\mathrm{d}t\otimes \delta_{\bar{u}}\otimes \delta_{\vs}$ 
for some fixed
$\bar{u}$ and ${\vs}$; in other words we fix the impact and the strength of selection
in each event. A general result can be obtained from our calculations
below by integration. 

In order to obtain a diffusion approximation, we shall speed up the rate of
reproduction events by 
a factor $n$, but scale down both the impact and the strength of the selection.
We write $u_n$, ${\vs}_n$ for the impact and strength of selection at the
$n$th stage of our rescaling.
We shall also scale $\Pi^{env}$ to have rate $n^\gamma$, 
with $\gamma>0$ to be chosen.
We shall need the joint generator $\mathcal{L}^{(n)}$ of the pair $(p,\zeta)$ 
at 
the $n$th stage of this rescaling.
We write $\pi$ for the uniform measure on $\{-1,1\}$ and $\IE_\pi$ for the
corresponding expectation. 
In an obvious notation, for suitable test functions $f$, we have
\begin{multline*}
\mathcal{L}^{(n)}f(p,\zeta) 
= \ind_{\{\zeta=-1\}}n
\left(\left[\frac{(1+{\vs}_n)p}{1 + {\vs}_np}\right]f\left((1-u_n)p + u_n,\zeta\right) 
\right.
\\
\left.
\phantom{blablablablablablablablablabla}
+  
\left[\frac{1-p}{1 + {\vs}_np}\right]f\left((1-u_n)p,\zeta\right) - f(p,\zeta)\right) 
\\ 
+  \ind_{\{\zeta=1\}}n\left(\left[\frac{p}{1+{\vs}_n(1-p)}\right]f\left((1-u_n)p + 
u_n,\zeta\right) 
\right.
\\
\left.
\phantom{blablablablablablablablablabla}
+  
\left[\frac{(1+{\vs}_n)(1-p)}{1+{\vs}_n(1-p)}\right]f\left((1-u_n)p,\zeta\right) - f(p,\zeta)\right)
\\ 
+n^{\gamma}\left(\IE_{\pi}[f(p,\cdot)]-f(p,\zeta)\right).
\end{multline*}
Expanding the ratios involving ${\vs}_n$ as geometric series, and using Taylor's
Theorem to expand $f$ (as a function of $p$), we obtain
\begin{align}\label{rescaled generratarot for NSLFV after Taylor expansion}
\mathcal{L}^{(n)}f(p,\zeta)&= n\lbrace p - {\vs}_{n}\zeta
p(1-p)
+ \mathcal{O}({\vs}_n^2)\rbrace \nonumber \\
&\qquad \qquad \times \lbrace u_{n}(1-p) f^{\prime}(p,\zeta) + 
\frac12 u_{n}^{2}(1-p)^{2}f^{\prime\prime}(p,\zeta) +\mathcal{O}(u_n^3)\rbrace 
 \nonumber\\
&+ n \lbrace(1-p) + {\vs}_{n}\zeta p(1-p)+\mathcal{O}({\vs}_n^2)
\rbrace \nonumber \\
&\qquad\qquad \times
\lbrace -u_{n}pf^{\prime}(p,\zeta) + \frac{1}{2}u_{n}^{2}p^{2} 
f^{\prime\prime}(p,\zeta)
+\mathcal{O}(u_n^3)\rbrace \nonumber \\ 
\nonumber
 &+ n^{\gamma}\left(\mathbb{E}_{\pi}[f(p,\cdot)]- f(p,\zeta)\right)\\
=&\frac12 n u_n^2 p(1-p)f^{\prime\prime}(p,\zeta)
- n {\vs}_nu_n\zeta p(1-p)f^{\prime}(p,\zeta) \nonumber\\
& + n^{\gamma}\left(\mathbb{E}_{\pi}[f(p,\zeta)]- f(p,\zeta)\right)
+\mathcal{O}\left( n ({\vs}_n^2u_n+u_n^2{\vs}_n+u_n^3)\right).
\end{align}
In order to obtain a diffusion limit, we see that we should take 
$nu_n^2$ to be $\mathcal{O}(1)$. If the environment didn't change, then
we would require $nu_n{\vs}_n$ to be $\mathcal{O}(1)$ and on passage to the
limit recover the classical Wright-Fisher diffusion with selection, whose 
generator, if $\zeta= -1$ say, takes the form
$$\mathcal{L}^{WFS}f(p)=\frac{1}{2}p(1-p)f^{\prime\prime}(p)+{\vs}p(1-p)f^{\prime}(p).$$
Since we are modelling short bursts of strong selection, we set
\begin{align}\label{scaling1}
u_{n} = n^{-\frac12} \bar{u} \quad {\vs}_{n} = n^{-\frac12+\alpha}{\vs},
\end{align}
for some $\alpha\in (0,1/4)$. The restriction $\alpha<1/4$ ensures that the 
error term $n{\vs}_n^2u_n$ in the 
expression~(\ref{rescaled generratarot for NSLFV after Taylor expansion})
is negligible as $n\to\infty$. 

We can then write the rescaled generator in the form
\begin{align}\label{LFVres}
\mathcal{L}^{(n)}f(p,\zeta) = \mathcal{L}^{neu}f(p,\zeta) + 
n^{\alpha}\mathcal{L}^{fsel}f(p,\zeta) + n^{\gamma}\mathcal{L}^{env}f(p,\zeta) + 
\mathcal{O}\left(n^{-\frac{1}{2} + 2\alpha}\right),
\end{align}
where 
\begin{align*}
\mathcal{L}^{neu}f(p,\zeta) &=  \frac12 \bar{u}^{2}p(1-p)f^{\prime\prime}(p,\zeta)\\
\mathcal{L}^{fsel}f(p,\zeta) &= -\zeta\bar{u}{\vs}p(1-p)f^{\prime}(p,\zeta) \\
\mathcal{L}^{env}f(p,\zeta) &= \mathbb{E}_{\pi}[f(p,\zeta)]- f(p,\zeta).
\end{align*}
To see how we should choose $\gamma$, we employ a `separation of timescales' trick due to 
 \cite{kurtz:1973}.
We apply the generator~\eqref{LFVres} to test functions of the form 
\begin{align}
\label{test_function_g}
g(p,\zeta) = f(p) + n^{-\delta}\left(\mathcal{L}^{fsel}f\right)(p,\zeta).
\end{align}
For this choice, we obtain
\begin{multline}\label{PH trick, non-spatial}
\mathcal{L}g(p,\zeta)
= \mathcal{L}^{neu}f(p) + n^{\alpha}\left(\mathcal{L}^{fsel}f\right)(p,\zeta) 
+ n^{-\delta}\mathcal{L}^{neu}\left(\mathcal{L}^{fsel}_{n}f\right)(p,\zeta) 
\\+ n^{\alpha -\delta} \mathcal{L}^{fsel}\left(\mathcal{L}^{fsel}f\right)(p,\zeta) - 
n^{\gamma-\delta}\left(\mathcal{L}^{fsel}f\right)(p,\zeta) 
+\mathcal{O}\left(n^{-\frac{1}{2} + 2\alpha}\right),
\end{multline}
where we have used the fact that $\mathcal{L}^{env}f(p)=0$ (since $f$ does not
depend on $\zeta$) and 
$\IE_\pi[n^{-\delta}\left(\mathcal{L}^{fsel}f\right)(p,\zeta)]=0$,
since
$\mathbb{E}_{\pi}[\zeta] = 0$.

Evidently, to obtain a non-trivial limit we should take $\gamma=\delta +\alpha$. 
The most interesting case is when $\delta=\alpha$ and so
$\gamma=2\alpha$. In that case, letting $n\to\infty$, in the 
limit the equation~(\ref{PH trick, non-spatial}) becomes  
\begin{equation}
\label{limgen}
\mathcal{L}f(p,\zeta)=\mathcal{L}^{neu}f(p) + 
\mathcal{L}^{fsel} \left(\mathcal{L}^{fsel}f\right)(p,\zeta).
\end{equation}
To evaluate the right hand side,
\begin{align*}
\mathcal{L}^{fsel}\left(\mathcal{L}^{fsel}f\right)(p,\zeta) 
&= \mathcal{L}^{fsel}\left(-\zeta\bar{u}{\vs}p(1-p)f^{\prime}(p)\right)
= \zeta^2\bar{u}^2{\vs}^2 p(1-p)\frac{d}{dp}\left(p(1-p)\frac{d}{dp}f(p)\right)\\
&= \zeta^{2}\bar{u}^{2}{\vs}^{2}p(1-p)(1-2p)f^{\prime}(p) + 
\zeta^{2}\bar{u}^{2}{\vs}^{2}p^{2}(1-p)^{2}f^{\prime\prime}(p). 
\end{align*}
Noting that $\zeta^2\equiv 1$, equation~(\ref{limgen}) then reads
\begin{align*}
\mathcal{L}f(p) =  \bar{u}^{2}{\vs}^{2}p(1-p)(1-2p)f^{\prime}(p) 
+ \left(\frac12 \bar{u}^{2}p(1-p) + \bar{u}^{2}{\vs}^{2}p^{2}(1-p)^{2}\right)
f^{\prime\prime}(p).
\end{align*}

\begin{remark}
There are other limits that can be obtained when $\gamma>2\alpha$. For example
if $\alpha=1/4$ and we resample the environment at every reproduction event, 
corresponding to $\gamma>1$, then \cite{miller:2012}
shows that, under the same scaling of $u_n$, the frequency of type $a$
alleles in the population converges weakly to the solution of
$$\mathrm{d}p=\frac{1}{2}{\vs}^2\bar{u}p(1-p)(1-2p)\mathrm{d}t+\bar{u}\sqrt{p(1-p)}\mathrm{d}B_t$$
for a standard Brownian motion $\{B_t\}_{t\geq 0}$. The deterministic drift
here arises from the term of order $nu_n{\vs}_n^2$ that under our previous
scaling we were able to neglect 
in~(\ref{rescaled generratarot for NSLFV after Taylor expansion}).
\end{remark}

Based on these calculations, the following proposition follows 
easily from Theorem~2.1  of \cite{kurtz:1992}, which we recall later as Theorem~\ref{kurtzaveraging}. 
In the interests of 
space, we omit the details of the proof, which follows from exactly the
same arguments as those that we employ in the spatial setting.
\begin{proposition}
Let $\{p^{(n)}(t)\}_{t\geq 0}$ denote the (non-spatial)
Lambda-Fleming-Viot process of Definition~\ref{LFVFS}
in which $\Pi$ has intensity
$n\mathrm{d}t\otimes\delta{u_n}\otimes\delta_{\vs_n}$, where 
$$u_n=n^{-1/2}\bar{u}, \quad {\vs}_n=n^{-1/2+\alpha}{\vs}, \mbox{ and }
\tau^{env}=n^{2\alpha},$$
for some $\alpha\in (0,1/4)$.
Suppose further that the sequence of initial conditions 
$\{p^{(n)}(0)\}_{n\geq 1}$ converges  
to $p_0$ as $n\to\infty$.
Then as $n$ tends to infinity, $\{p^{(n)}(t)\}_{t\geq 0}$
converges weakly 
in $\mathcal{D}([0,\infty), [0,1])$  
(the space of c\`adal\`ag functions taking values in $[0,1]$)
to the one-dimensional diffusion with drift
\begin{align*}
\bar{u}^{2}{\vs}^{2}p(1-p)(1-2p)
\end{align*}
and quadratic variation
\begin{align*}
\bar{u}^{2}p(1-p) + 2\bar{u}^{2}{\vs}^{2}p^{2}(1-p)^{2},
\end{align*}
started from $p_0$.
In other words, the limiting process is the unique weak solution to the equation
\begin{multline} \label{non-spatial SPDE approximation}
\mathrm{d}p(t) =  \bar{u}^{2}{\vs}^{2}p(t)(1-p(t))(1-2p(t))\mathrm{d}t \\
+ 
\bar{u}\sqrt{p(t)(1-p(t))}\mathrm{d}B_t^1 + \sqrt{2}\bar{u}{\vs}p(t)(1-p(t))\mathrm{d}B_t^2,
\end{multline}
with $p(0)=p_0$, and $\{B^1_{t}\}_{t\geq 0}$, $\{B_t^2\}_{t\geq 0}$ 
independent standard Brownian Motions.
\end{proposition}

\section{Definition and scaling of the SLFVFS}
\label{results}

In this section we first extend the Lambda-Fleming-Viot model with 
fluctuating selection
of Section~\ref{nospace}
to the spatial setting. The idea is simple: reproduction events are still
driven by a Poisson point process, but now, in addition to specifying the 
strength of selection and the impact associated with each event, we must also
specify the spatial region in which it takes place. As has become
usual in this framework, we shall take those regions to be closed balls (indeed
for simplicity we shall take our events to be of a fixed radius), but
the same results will hold under much more general conditions, subject
to some symmetry and boundedness assumptions.
Having defined the model, we state our main scaling result for the
spatial model. 

\subsection{Spatial Lambda-Fleming-Viot process with fluctuating selection}

We suppose that the population, which is distributed across $\IR^d$, is
subdivided into two genetic types $\{a,A\}$.
 As explained in detail in
 \cite{etheridge/veber/yu:2014}, which in turn borrows results
from \cite{veber/wakolbinger:2015}, formally, at each time the state of the
population is described by a measure $M_t$ 
on $\IR^d\times K$, where $K=\{a,A\}$,
whose first marginal is Lebesgue measure on $\IR^d$. This
space of measures, which we denote by ${\cal M}_\lambda$, is equipped with
the topology of vague convergence, under which it is compact. 
At any fixed time there is a density
$w(t,\cdot) : \IR^d\rightarrow [0,1]$ such that 
$$M_t(dx,d\kappa)=\left(w(t,x)\delta_a(d\kappa)+
(1-w(t,x))\delta_A(d\kappa)\right) \mathrm{d}x.$$
Of course $w(t,x)$, which one should interpret as the proportion of the 
population at the location $x$ at time $t$ that is of type $a$,
is only defined up to a Lebesgue null set. In what
follows, we shall consider a representative of the density of $M_t$. It
will be convenient to fix a representative $w(0,\cdot)$ of $M_0$ and then 
update it using the procedure described in the definition below, but the
reader should bear in mind that the fundamental object is the 
measure-valued evolution. This becomes important when we talk about 
convergence of our rescaled processes; tightness will be immediate in the
space of measures, but we will need to work harder to identify the 
dynamics of the density of the limit.

In what follows, for every 
$f \in C_{c}$ (continuous functions of compact support on $\IR^d$)
we shall use the notation
\begin{align*}
\left\langle w,f\right\rangle = \int_{\IR^d} w(x)f(x)\mathrm{d}x.
\end{align*}

Recall that the limit that we obtained in Section~\ref{nospace} 
corresponded to our throwing away the terms of order
$n{\vs}_n^2u_n$ in~(\ref{rescaled generratarot for NSLFV after Taylor expansion}).
In other words we approximated 
$(1+{\vs})p/(1+{\vs}p)$ by $p+{\vs}p(1-p)=
p(1-{\vs})+{\vs}\left(1-(1-p)^2\right)$ and similarly $p/(1+{\vs}(1-p))$ was 
approximated by $p-{\vs}p(1-p)= p(1-{\vs})+{\vs}p^2$. Under this
approximation, since reproduction events are
based on a Poisson process of events, we can think of splitting those events
into two types: neutral events and selective events. In the non-spatial
setting, neutral events occur at rate $(1-{\vs})$ and, for such an event, 
the chance that the parent is type $a$ is $p$. Selective events
fall at rate ${\vs}$. One then selects two `potential' parents. If $\zeta=-1$,
then the offspring are type $a$ provided not both potential parents are 
type $A$, which has probability $1-(1-p)^2$,
whereas if $\zeta=1$, the offspring are type $a$ only if both
potential parents are type $a$ (probability $p^2$).

To avoid additional algebra, we shall define the spatial version of our
model using this approximation. In our main scaling result, 
we shall indeed choose our scaling in such a
way that $n{\vs}_n^2u_n\to 0$ as $n\to\infty$.

\begin{definition}[Spatial Lambda-Fleming-Viot process with fluctuating selection (SLFVFS)]
\label{SLFVSFE_def}
Let $\mu$ be a measure on $(0, \infty)$ and for each $r\in (0,\infty)$, 
let $\nu_{r}$ be a probability measure on $[0,1]$, such that the 
mapping $r \rightarrow \nu_r$ is measurable and 
\begin{align}
\label{integrability}
\int_{(0,\infty)}r^{d}\int_{[0,1]}u\; \nu_{r}(\mathrm{d}u)\mu(\mathrm{d}r) < \infty.
\end{align} 
Further, fix ${\vs}\in [0,1]$ and 
let $\Pi^{neu}$, $\Pi^{fsel}$,  be independent Poisson point processes on 
$\IR_+\times \IR^d \times (0,\infty) \times [0,1]$ with intensity measures 
$(1-{\vs})\mathrm{d}t \otimes \mathrm{d}x \otimes \mu(\mathrm{d}r)\nu_{r}(\mathrm{d}u)$
and 
${\vs}\mathrm{d}t \otimes \mathrm{d}x \otimes \mu(\mathrm{d}r)\nu_{r}(\mathrm{d}u)$ 
respectively. 

Let $\Pi^{env}$ be a Poisson process, independent of $\Pi^{neu}$, $\Pi^{fsel}$, 
with intensity $\tau^{env}$, dictating the times of the changes in the environment.
Let $\{\xi^{(m)}(\cdot)\}_{m\geq 0}$ be a  family of identically distributed 
random fields  such that 
\begin{align*}
\mathbb{P}\left[\xi^{(m)}(x) =-1  \right] = \frac{1}{2}=
\mathbb{P}\left[\xi^{(m)}(x) =+1 \right],\\
\mathbb{E}\left[ \xi^{(m)}(x)\xi^{(m)}(y) \right] = g(x,y),
\end{align*}
where the covariance function
$g(x,y)$ is an element of $C_{b}\left(\IR^d \times \IR^d\right)$.
Set $\tau_0=0$ and write $\{\tau_m\}_{m\geq 1}$ for the points in $\Pi^{env}$ and define
$$\zeta(t,\cdot):=\sum_{m=0}^\infty\xi^{(m)}(\cdot)\ind_{[\tau_m,\tau_{m+1})}(t).$$
In other words, the environment $\zeta(t,\cdot)$ is resampled, independently, at
the times of the Poisson process $\Pi^{env}$.

The {\em spatial Lambda-Fleming-Viot process with fluctuating selection (SLFVFS)}
with driving noises $\Pi^{neu}$, $\Pi^{fsel}$, $\Pi^{env}$, is the 
$\mathcal{M}_{\lambda}$-valued process $M_t$ with
dynamics described as follows. Let $w(t_{-},\cdot)$ be a representative of the 
density of $M_{t_-}$ immediately before an event $(t,x,r,u)$ from $\Pi^{neu}$
or $\Pi^{fsel}$. Then the measure $M_{t_-}$ immediately after the event has 
density $w(t, \cdot)$ determined by:
\begin{enumerate}
\item If $(t,x,r,u) \in \Pi^{neu}$, a neutral event occurs at time 
$t$ within the closed ball $B(x,r)$. Then
\begin{enumerate}
\item Choose a parental location $l$ according to the uniform distribution on
$B(x,r)$. 
\item Choose the parental type $\kappa \in \{a,A\}$ according to the distribution
\begin{align*}
\mathbb{P}\left[\kappa = a \right] = w (t_{-},l) , 
\quad \mathbb{P}\left[\kappa = A \right] = 1 -  w (t_{-},l).
\end{align*}
\item A proportion $u$ of the population within
$B(x,r)$ dies and is replaced by offspring with type $\kappa$.
Therefore, for each point $y \in B(x,r)$, 
\begin{align*}
w(t,y) = w(t_{-},y)(1 - u) + u\ind_{\{\kappa = a\}}. 
\end{align*}
\end{enumerate}

\item If $(t,x,r,u) \in \Pi^{fsel}$, a   selective event  occurs at 
time $t$ within the closed ball $B(x,r)$. Then
\begin{enumerate}
\item Choose two parental locations $l_{0}, l_{1}$ independently, 
according to the uniform distribution on $B(x,r)$.
\item Choose the two parental types, $\kappa_{0}, \kappa_{1}, $ independently, according to  
\begin{align*}
\mathbb{P}\left[\kappa_{i} = a \right] = w (t_{-},l_{i}), \quad 
\mathbb{P}\left[\kappa_{i} = A \right] = 1 -  w (t_{-},l_{i}).
\end{align*}
\item A proportion $u$ of the population within $B(x,r)$ dies
and is replaced by offspring with type chosen as follows:
\begin{enumerate}
\item If  $\zeta(t,x) =1$, their type is set to be $a$ if   
$\kappa_{0} = \kappa_{1} = a$, and $A$ otherwise. 
Thus for each $y \in B(x,r)$ 
\begin{align*}
w(t,y) & = (1 - u)w(t_{-},y)  + u\ind_{\{\kappa_{0} = \kappa_{1} = a\}}. 
\end{align*}
\item If $\zeta(t,x) =-1$, their type is set to be $a$ if   
$\kappa_{0} = \kappa_{1} = a$ or $\kappa_{0} \neq \kappa_{1}$ and $A$ otherwise, 
so that for each $y \in B(x,r)$, 
\begin{align*}
w(t,y) & = (1 - u)w(t_{-},y)  + u\left( \ind_{\{\kappa_{0} = \kappa_{1} = a\}} + 
\ind_{\{\kappa_{0} \ne \kappa_{1}\}} \right). 
\end{align*} 
\end{enumerate}
\end{enumerate}
\end{enumerate}
\end{definition}
We have tacitly assumed that $w(t_{-},l_i)$ is defined. It is, with probability
one, so we declare that if it is not defined, then we resample and try again.
For a construction of the random fields $\{\xi^{(m)}(\cdot)\}_{m\geq 0}$
of Definition~\ref{SLFVSFE_def} we refer to \cite{ma:2009}, 
especially their Example~1. The arguments presented there require only 
minor adaptation. 

Existence of the SLFVFS is guaranteed by the methods of \cite{etheridge/veber/yu:2014}.
Indeed, we could have taken different measures $\mu$ and $\nu_r$ according to 
whether events are
selective or neutral. Although it is convenient to take the strength of 
selection to be constant in space and have its direction determined by 
the variable
$\zeta\in\{-1,+1\}$, we could, of course, have defined a much more 
general model. 
For example, one could allow ${\vs}$ to vary in space, or even resample
${\vs}\zeta$ from a suitable random field whenever the environment is
resampled. However, this would be at the expense of considerably more 
complicated notation and it would become more involved to exploit the 
Poisson structure of our model. 
See Remark~\ref{remark on environment} below for some comments on when 
our scaling result would generalise.

One of the key tools in the study of the neutral SLFV 
is the dual process of coalescing 
random walkers which traces out the genealogical trees 
relating individuals in a sample
from the population.  An ancestral 
lineage doesn't move until it is both 
in the region affected by an event and is among the 
offspring of that event, 
at which time it jumps to the location of the parent of 
the event (which is uniformly distributed on the affected region).
Things are more complicated in the presence of selection.
Whereas in the neutral case we can always identify the distribution of the 
location of the parent of each event, 
now, at a selective event, even knowing the state of the environment, 
we are unable to identify which of the `potential parents' is the
true parent of the event without knowing their types. These can only be established by
tracing further into the past.
The resolution is to follow all
potential ancestral lineages backwards in time. This results in a system of 
branching and coalescing walks in which branching and coalescence events
are `marked' according to the state of the environment at the time
at which they occur.

Just as in the neutral case, the
dynamics of the dual are driven by the time reversals 
$\overleftarrow{\Pi}^{neu}$, $\overleftarrow{\Pi}^{sel}$, 
$\overleftarrow{\Pi}^{env}$ 
of the Poisson point processes of events 
that drove the forwards in 
time process, that is
\begin{align*}
\overleftarrow{\Pi}^{i} := \left\lbrace
(-t,x,r,u): (t,x,r,u) \in \Pi^{i}
\right\rbrace,
\quad
i \in \{\text{neu}, \text{sel}, \text{env}\}
.
\end{align*}
The distribution of these Poisson point processes is 
invariant under the time reversal.

We emphasize that time for the process of ancestral lineages
runs in the opposite direction to that for the allele frequencies. 
Our dual will relate the distribution of allele frequencies in a 
sample from the population at a time $T$, to 
allele frequencies at time $0$.
More precisely,
suppose that we know the frequencies $w(0,\cdot)$ of $a$-alleles at time $0$.
At time $T$, which we think of as `the present',
we sample $j$ individuals from locations 
$\chi_0^1,\ldots ,\chi_0^j$. Tracing backwards in time, 
we write $\chi_s^1,\ldots ,\chi_s^{N_s}$ 
for the locations of the $N_s$ {\em potential ancestors} that make up our 
dual at time $s$ before the present. 
\begin{definition}[Ancestral selection graph]
\label{dualprocessdefn}
We first define a
$\bigcup_{n\geq 1}(\IR^d)^n$-valued Markov process,
$((\chi^i_t)_{i=1}^{N_t})_{t\geq 0}$, enriched by `environmental marks' 
as follows:

At each $\tau\in \overleftarrow{\Pi}^{env}$ the environment is resampled;
  
At each event $(t,x,r,u)\in \overleftarrow{\Pi}^{neu}$,
\begin{enumerate}
\item for each $\chi_{t_-}^i\in B(x,r)$, independently mark 
the corresponding 
potential ancestor with probability $u$;
\item if at least one lineage is marked, 
all marked lineages disappear and are replaced by a single
potential ancestor, whose location is drawn uniformly at random 
from within $B(x,r)$.
\end{enumerate}
At each event $(t,x,r,u)\in \overleftarrow{\Pi}^{sel}$:
\begin{enumerate}
\item for each $\chi_{t_-}^i\in B(x,r)$, independently mark the corresponding
potential ancestor with probability $u$;
\item if at least one lineage is marked,
all marked lineages disappear and are replaced by {\em two} potential
ancestors, whose locations are drawn independently and uniformly from
within $B(x,r)$. The type of the environment is recorded.
\end{enumerate}
In both cases, if no particles are marked, then nothing happens.

To determine the distribution of types of a sample of the 
population $w(T,\cdot)$, taken from locations 
$\chi_0^1, \ldots, \chi_0^j$, knowing the distribution of $w(0,\cdot)$ at
time $T$ before the present, 
first evolve the process of branching and coalescing lineages until
time $T$.
At time $T$ assign types to $\chi_T^1,\ldots , \chi_T^{N_T}$ using
independent Bernoulli random variables such that 
$\IP[\mathtt{Type}(\chi_T^i)=a]=w(0,\chi_T^i)$. Tracing back through the system of
branching and coalescing lineages $(\chi^i_t)_{i=1}^{N_t}$, 
we define types recursively:
at each neutral event the lineages that coalesced during the event are assigned the
type of the parent; at a selective event, if $Z=1$ then all coalescing lineages are
type $a$ if and only if both parents are type $a$, otherwise they are type $A$ whereas
if $Z=-1$, all coalescing lineages are type $A$ if and only if both parents are type $A$,
otherwise they are type $a$.  The distribution of types at time zero is the desired
quantity.
\end{definition}
Since we only consider finitely many initial individuals in the sample, 
the jump rate in this process is finite and so this description gives rise to
a well-defined process. 

This dual process is the analogue for the SLFVFS of the Ancestral Selection Graph (ASG), 
introduced in the companion papers\cite{krone/neuhauser:1997},   \cite{neuhauser/krone:1997},
which describes all the potential ancestors 
of a sample from a population evolving according to the Wright-Fisher diffusion 
with selection. Indeed we could be more careful and use this process to extract the 
genealogy of a sample from the population. However, in this setting, 
this object seems to be rather
unwieldy and, under the scalings in which we are interested, it will not 
converge to a well-defined limit.

\begin{remark}
Informally, the procedure described above allows us to write down 
an expression, in terms of the marked process of branching and 
coalescing ancestral lineages and $w(0,\cdot)$,
for $\IE[\prod_{i=1}^jw(T,\chi_0^i)]$; that is the 
probability that $j$ individuals, sampled from the present day 
population at locations $\chi_0^1,\ldots, \chi_0^j$, are all of type $a$.  
More formally, since the density of the SLFVFS is only defined 
Lebesgue-almost everywhere, 
the quantities $w(T,\chi_0^i)$ are 
only defined for Lebesgue almost every choice of  
$\chi_0^1, \ldots ,\chi_0^{j}$ and, just as in \cite{etheridge/veber/yu:2014} Section~1.2,
this duality must be defined `weakly', that is by integrating against 
a suitable test function $\psi(\chi_0^1,\ldots, \chi_0^j)$. Also mirroring that 
setting, the resulting `moment duality'
is sufficient to guarantee uniqueness of the SLFVFS. 
Since we do not use the duality in what follows, we refer the reader 
to \cite{etheridge/veber/yu:2014} 
for details.
\end{remark}

\subsection{Scaling the SLFVFS}\label{statement of main result}

We are interested in the effects of fluctuating selection over large spatial
and temporal scales and so we shall consider a rescaling of our model.
 \cite{etheridge/veber/yu:2014} consider the corresponding process in which
selection does not fluctuate with time, but instead always favours type 
$A$ (say). In that setting it is shown that if impact scales as $u_n=\bar{u}/n^{1/3}$
and selection scales as ${\vs}_n={\vs}/n^{2/3}$, then as $n\to\infty$, 
$w(nt, n^{1/3}x)$ (or rather a local average of this quantity) converges
in $d\geq 2$ to the solution to the deterministic Fisher-KPP equation,
and to the solution of 
the corresponding stochastic p.d.e.~in which a `Wright-Fisher noise'
term, corresponding
to genetic drift, has been added in $d=1$. Here we wish to consider short
periods of stronger selection and so, by analogy with what we did in 
Section~\ref{nospace}, we choose ${\vs}_n={\vs}n^\alpha/n^{2/3}$ 
for some $\alpha>0$,
but we change the favoured type at times of mean $1/n^{2\alpha}$. This is
of course the scaling suggested by the Central Limit Theorem (and is the
natural analogue of our results in Section~\ref{nospace}).
In order to be able to ignore terms of order $n{\vs}_n^2u_n$, we must now 
take $\alpha\in (0,1/6)$.

We must also scale the environment in a consistent way. 
In the examples that we have in mind, environmental
correlations can be expected to extend over very large scales and so we
actually fix the correlations in the {\em limiting} environment by fixing
the distribution of a random field $\xi$ and at the $n$th stage of the 
scaling sampling the environment according to
$\xi_n$ determined by
\begin{align*}
\xi(x) = \xi_n(n^{-\frac{1}{3}}x).
\end{align*}
At the $n$th stage of the rescaling, the environment will be resampled
at points of a rate $n^{2\alpha}$ Poisson process.
\begin{remark}[Extensions]
\label{remark on environment}
We have taken selection to be constant in magnitude and just to vary in 
sign. This is not necessary, even for our scaling result. As
an obvious extension, we could fix the distribution
of  ${\vs}(x)\xi(x)$ and, at the
$n$th stage of the scaling define 
\begin{equation}
\nonumber
{\vs}_n(n^{-1/3}x)\xi_n(n^{-1/3}x)=\frac{n^\alpha}{n^{2/3}}\vs (x)\xi(x),
\end{equation}
At the expense of introducing an additional truncation of 
$\vs(x)$ at the
$n$th stage, to ensure that the $|\vs_n(x)\xi_n(x)|<1$,
it is enough to insist that
$$\IP[\vs(x)\xi(x)>z]=\IP[\vs(x)\xi(x)<-z], \quad \forall z\in\IR,$$
and
$$\IE[|\vs(x)\xi(x)|^{2+\epsilon}]<\infty, \qquad\mbox{for some }\epsilon>0,$$
plus some regularity to reflect equation~(\ref{regularity of g}) below.
To avoid a proliferation of notation, we omit this somewhat artificial
generalisation of our results.
\end{remark}
Our definition of the SLFVFS is still
rather general. We include it to underline the 
possibility of extending our results. However, in the interest of avoiding
even more complex expressions than those that follow, from now on we 
shall specialise to fix the radius and impact of reproduction events. 
\begin{assumption}
From now on, fix $R\in (0,\infty)$ and $\bar{u}\in (0,1)$ and take
$$ \mu({dr}) = \delta_{R}(\mathrm{d}r), 
\qquad \nu_{r}(\mathrm{d}u) = \delta_{\bar{u}}(\mathrm{d}u).$$
\end{assumption}

Just as in \cite{etheridge/veber/yu:2014}, we shall prove convergence, not
of the sequence of densities of the SLFVFS, but of a sequence of local averages.
We require some notation.
Let $R$ be the fixed radius of events. Set
\begin{align*}
R_{n}  = n^{-\frac{1}{3}}R, \quad B_{n}(x) = B(x,R_{n}),
\end{align*}
and 
define the sequence of rescaled processes  
\begin{align}\label{scaling3}
w_{n}(t,x) = w(nt,n^{\frac13}x), \quad \overline{w}_{n}(t,x) = \vint_{B_{n}(x)}w_{n}(t,y)\mathrm{d}y,
\end{align}
where $\vint_{B_{n}(x)}$ denotes an average integral over 
the ball $B_{n}(x)$.
We write $V_R$ for the volume of a ball of radius $R$.

\begin{theorem} \label{scaling result}
Write $(\overline{M}_{t}^{n})_{t\geq 0}$ for the measure-valued process with
density  $(\overline{w}_{n}(t))_{t\geq 0}$. Suppose that $(M_0^n)_{n\geq 1}$
converges weakly in ${\cal M}_{\lambda}$ 
to the measure $M_{0}^{\infty}$ with density 
$w^{\infty}(0,x) = \lim_{n\rightarrow\infty}\overline{w}_n(0,x)$. Further, fix $\alpha\in 
(0,1/6)$, set $\vs_n=\vs n^{\alpha}/n^{2/3}$, $u_n=\bar{u}/n^{1/3}$ and suppose that the environment
is resampled at the times of a Poisson process of rate $n^{2\alpha}$. We assume
that the correlation function $g(x,y)$ that determines the environment
satisfies
\begin{equation}
\label{regularity of g}
|g(x,x)-g(x,y)|\leq C|x-y|\qquad\mbox{ for all }x,y\in\IR^d.
\end{equation}
Then the sequence
$(\overline{M}_\cdot^n)_{n\geq 1}$ is tight in 
$D([0,\infty), {\cal M}_\lambda)$ (the space of c\`adl\`ag functions on $[0,\infty)$ taking values in  ${\cal M}_\lambda$). Moreover, for any
weak limit point $(M_t^\infty)_{t \geq 0}$, writing $w^\infty$ for a 
representative of the density of $M^\infty$,
\begin{enumerate}
\item for dimension $d = 1$, $w^{\infty}$ is the  process for which, 
for every $F \in C_{c}^{\infty}\left( \mathbb{R}\right)$ and for 
every $f \in C_{c}^{2}\left( \mathbb{R}\right)$,
\begin{multline}\label{Martingale Problem for original model 1d}
F\left(\langle w^{\infty}(t),f \rangle\right) - F\left(\langle w^{\infty}(0),f \rangle\right) \\
-\int_{0}^{t} F^{\prime}\left(\langle w^{\infty}(s),f \rangle\right)
\left\lbrace\left\langle w^{\infty}(s),\frac{\bar{u}\Gamma_R}{2}\Delta f\right\rangle \right.
 \phantom{V_R^2\bar{u}^2{\vs}^{2}\langle {w}^{\infty}(s)(1-{w}^{\infty}(s))\langle w^{\infty}(s)}
 \\
\left. \phantom{\left\langle \frac{\Gamma_R}{2}\right.}
 + V_R^2\bar{u}^2{\vs}^{2}\langle {w}^{\infty}(s)(1-{w}^{\infty}(s))
(1-2{w}^{\infty}(s)),f \rangle \right\rbrace \mathrm{d}s\\
- \int_{0}^{t} F^{\prime\prime}\left(\langle w^{\infty}(s),f \rangle\right) \left\lbrace
\phantom{\int_{\mathbb{R}^{d}}\int_{\mathbb{R}^{d}} V_R^2\bar{u}^{2}{\vs}^{2} g(x,y)}\right.
\phantom{\int_{\mathbb{R}^{d}}
w^{\infty}(s,y)w^{\infty}(s,y)}
\\
\left.\int_{\mathbb{R}^{d}}\int_{\mathbb{R}^{d}} \left[V_R^2\bar{u}^{2}{\vs}^{2} g(x,y) w^{\infty}(s,x)(1 - w^{\infty}(s,x)) \right.\right. \phantom{\int_{\mathbb{R}^{d}}}
\\
\left.\left. \phantom{V_{R}}\phantom{\int_{\mathbb{R}^{d}}}
\times w^{\infty}(s,y)(1 - w^{\infty}(s,y))f(x)f(y) \right]\mathrm{d}x\mathrm{d}y \right. \\
 \left. \phantom{\int_{\mathbb{R}^{d}}} + \frac{\bar{u}^2V_{R}^2}{2}\langle w^{\infty}(s)(1 - w^{\infty}(s)),f^2 \rangle \right\rbrace \mathrm{d}s
\end{multline}
is a martingale;
\item for dimension $d \geq 2$,  
for every  $F \in C_{c}^{\infty}\left( \mathbb{R}\right)$ and 
for every $f \in C_{c}^{2}\left( \mathbb{R}^{d}\right)$,
\begin{multline}\label{Martingale Problem for original model nd}
F\left(\langle w^{\infty}(t),f \rangle\right) - F\left(\langle w^{\infty}(0),f \rangle\right) \\
-\int_{0}^{t} F^{\prime}\left(\langle w^{\infty}(s),f \rangle\right)
\left\lbrace\left\langle w^{\infty}(s),\frac{\bar{u}\Gamma_R}{2}\Delta f\right\rangle \right.
 \phantom{V_R^2\bar{u}^2{\vs}^{2}\langle {w}^{\infty}(s)(1-{w}^{\infty}(s))\langle w^{\infty}(s)}
 \\
\left. \phantom{\left\langle \frac{\Gamma_R}{2}\right.}
 + V_R^2\bar{u}^2{\vs}^{2}\langle {w}^{\infty}(s)(1-{w}^{\infty}(s))
(1-2{w}^{\infty}(s)),f \rangle \right\rbrace \mathrm{d}s\\
- \int_{0}^{t} F^{\prime\prime}\left(\langle w^{\infty}(s),f \rangle\right) \left\lbrace
\phantom{\int_{\mathbb{R}^{d}}\int_{\mathbb{R}^{d}} V_R^2\bar{u}^{2}{\vs}^{2} g(x,y)}\right.
\phantom{\int_{\mathbb{R}^{d}}
w^{\infty}(s,y)w^{\infty}(s,y)}
\\
\left.\int_{\mathbb{R}^{d}}\int_{\mathbb{R}^{d}} \left[V_R^2\bar{u}^{2}{\vs}^{2} g(x,y) w^{\infty}(s,x)(1 - w^{\infty}(s,x)) \right.\right. \phantom{\int_{\mathbb{R}^{d}}}
\\
\left.\left. \phantom{V_{R}}\phantom{\int_{\mathbb{R}^{d}}}
\times w^{\infty}(s,y)(1 - w^{\infty}(s,y))f(x)f(y) \right]\mathrm{d}x\mathrm{d}y \right\rbrace \mathrm{d}s
\end{multline}
is a martingale.  Moreover, the solution to this martingale problem is unique 
and so $\{\overline{M}^n\}_{n\geq 1}$ actually converges.
\end{enumerate}
The constant $\Gamma_R$ depends only on $R$ and is defined 
in~\eqref{defn of Gamma}.
\end{theorem}
The proof of uniqueness in $d\geq 2$ uses a pathwise  uniqueness result of  \cite{rippl/sturm:2013}
for a corresponding stochastic p.d.e.. In Appendix~\ref{mgpspde}, we follow
the approach of  \cite{kurtz:2010}, which uses the Markov Mapping 
Theorem, to show that any solution to the
martingale problem~(\ref{Martingale Problem for original model nd})
is actually a weak solution to the stochastic p.d.e.:
\begin{multline}\label{Limiting equation nd}
\mathrm{d}w^{\infty}  = 
\left(\frac{\bar{u}\Gamma_R}{2}\Delta w^{\infty} + 
\bar{u}^{2}V_R^2{\vs}^{2}w^{\infty}(1 - w^{\infty})(1 - 2w^{\infty})\right)\mathrm{d}t 
\\+   
\sqrt{2}\bar{u}V_R{\vs}w^{\infty}(1 - w^{\infty})W(\mathrm{d}t,\mathrm{d}x), 
\end{multline}
where the noise $W$ is white in time and coloured in space,
with quadratic variation given by 
\begin{equation}
\langle W(\phi)\rangle_t = t\int_{\mathbb{R}^{d}}\int_{\mathbb{R}^{d}}g(x,y)\phi(x)\phi(y)\mathrm{d}x\mathrm{d}y.
\end{equation}
The corresponding equation for dimension $d=1$ is
\begin{multline}\label{Limiting equation 1d}
dw^{\infty}  = \left(\frac{\bar{u}\Gamma_R}{2}
\Delta w^{\infty} + 4R^2 \bar{u}^{2}{\vs}^{2}w^{\infty}(1 - w^{\infty})
(1 - 2w^{\infty})\right)dt
\\
 +   
\sqrt{2}2R\bar{u}{\vs}w^{\infty}(1 - w^{\infty})W(\mathrm{d}t,\mathrm{d}x) + 
2R\bar{u}\sqrt{w^{\infty}(1 - w^{\infty})}{\mathcal{W}}(\mathrm{d}t,\mathrm{d}x), 
\end{multline}
where $W$ is white in time and coloured in space as above and 
${\mathcal{W}}$ is a space-time white noise.
As in $d\geq 2$, any solution to the martingale 
problem~(\ref{Martingale Problem for original model 1d})
will be a weak solution to this stochastic p.d.e., but we only have a 
proof of uniqueness of~(\ref{Limiting equation 1d})
in the special case in which
${W}$ is also space-time white noise (in which case we can 
invoke the duality of Section~\ref{duality}).

\section{Duality}
\label{duality}

In general, we have been unable to identify a useful dual process for our
limiting equation. The exceptions are the non-spatial setting and the special 
case of one spatial dimension, with both noises in the stochastic p.d.e.~being
white in space as well as time. In order to obtain these duals, we transform
the system in a way inspired by \cite{blath/etheridge/meredith:2007}.

Consider first the non-spatial case. We rewrite 
equation~\eqref{non-spatial SPDE approximation} by
making the substitution $X=1-2p$. 
The process $X$, which
takes values in $[-1,1]$, satisfies
\begin{align}
\mathrm{d}X_{t} =
\frac12\bar{u}^{2}{\vs}^{2}\left(X_t^{3} - X_t\right)\mathrm{d}t
+\bar{u}\sqrt{1 - X_{t}^{2}} \mathrm{d}B_t^1 + \frac{\sqrt{2}}{2}\bar{u}{\vs}
(1 - X_{t}^{2})\mathrm{d}B_{t}^2, 
\label{transformed sde}
\end{align}
for independent Brownian motions $B^1$, $B^2$.

\begin{lemma}
The solution to the transformed equation~\eqref{transformed sde} is 
dual to a branching annihilating process $\{N_t\}_{t\geq 0}$ with
transitions 
\begin{enumerate}
\item $N \mapsto N+ 2$ at rate 
\begin{align*}
\frac{\bar{u}^2{\vs}^2}{2}\left(N+\binom{N}{2}\right);
\end{align*}
\item For $N\geq 2$, $N\mapsto N-2$ at rate
\begin{align*}
\bar{u}^{2}(1+\frac{{\vs}^{2}}{2})\binom{N}{2}. 
\end{align*}
\end{enumerate}
The duality relationship takes the form
\begin{align*}
\mathbb{E}_{X_0}\left[X_{t}^{N_{0}} \right] = 
\mathbb{E}_{N_0}\left[X_{0}^{N_{t}} \right]
\end{align*}
where the expectation on the left is with respect to the law of
$\{X_t\}_{t\geq 0}$ started from initial condition $X_0$,
and that on the right is with respect to the 
law of $\{N_t\}_{t\geq 0}$, started from $N_0$.
\end{lemma} 
The proof is an application of It\^o's formula.
It is easy to see that started from an even number of particles,
the dual process will die out in finite time,
(count the number of pairs of particles and compare to
a subcritical birth-death process),
corresponding to the process $\{p_t\}_{t\geq 0}$ of allele frequencies
being absorbed in either zero or one. Of course, this is 
also readily checked directly for the diffusion~(\ref{basic sde})
using the theory of speed and scale, 
but if we could find an analogous dual for a spatially extended
population, where the theory of one dimensional diffusions is no 
longer helpful, we might be able to exploit it to study the 
behaviour of allele frequencies.

In one spatial dimension, if the noises ${W}$ and ${\mathcal W}$
are both white in space as well as time, then we can extend this.
\begin{lemma}
\label{duality lemma}
Suppose that $d=1$ and $w^\infty$ solves
\begin{multline*}
\mathrm{d}w^{\infty}  = \left(\frac{\bar{u}\Gamma_R}{2}
\Delta w^{\infty} + 4R^2 \bar{u}^{2}{\vs}^{2}w^{\infty}(1 - w^{\infty})
(1 - 2w^{\infty})\right)\mathrm{d}t + \\   
\sqrt{2}2R\bar{u}{\vs}w^{\infty}(1 - w^{\infty})W(\mathrm{d}t,\mathrm{d}x) + 
2R\bar{u}\sqrt{w^{\infty}(1 - w^{\infty})}{\mathcal{W}}(\mathrm{d}t,\mathrm{d}x), 
\end{multline*}
where $W$ and ${\mathcal{W}}$ are independent space-time white noises.
Then setting $X_t(x)=1-2w^{\infty}(t,x)$, $\{X_t(x)\}_{t\geq 0}$ is dual
to a system of branching-annihilating Brownian particles 
whose spatial
locations at time $t$ we denote by 
$\chi_1,\ldots,\chi_{N_t}$, and whose dynamics are described as follows:
\begin{enumerate}
\item Each particle, independently, follows a Brownian motion in $\IR$, with
diffusion constant $\bar{u}\Gamma_R$;
\item Each particle, independently, splits into three at rate 
$\bar{u}\vs^2/2$;
\item Each pair of particles annihilates, at
a rate $\bar{u}^2(1+{\vs}^2/2)$, measured by their intersection local time;
\item Each pair of particles replicates (i.e.~is replaced by two 
identical pairs) at rate 
$\bar{u}^{2}{\vs}^2)/2$, also measured by their intersection local time.
\end{enumerate}
The duality is expressed for points $\chi_1(0),\ldots \chi_{N_0}(0)
\in \IR$ and any continuous function $w^\infty(0,\cdot):\IR\rightarrow [0,1]$,
through
\begin{equation}
\label{spatial duality}
\IE\left[\prod_{i=1}^{N_0}w^\infty_{t}(\chi_{i}(0))\right]=
\IE\left[\prod_{i=1}^{N_t}w^\infty_{0}(\chi_{i}(t))\right],
\end{equation}
where the expectation on the left is with respect to the law of the 
stochastic p.d.e.~and that on the right with respect to the law of 
the dual system of branching and annihilating lineages. 
\end{lemma}
 \cite{tribe:1995} gives a construction of the analogous system of 
coalescing Brownian motions, which is dual to the stochastic heat
equation with Wright-Fisher noise, as discussed in 
 \cite{shiga:1988}.
\cite{doering/mueller/smereka:2002} 
provide a complete derivation in that context;  
see also  \cite{liang:2009}. 
We also note that \cite{birkner:2003} considers a similar 
system of branching random walkers
on $\IZ^d$, in which particles reproduce at a rate that depends
on the number of other particles within the same site.  
\begin{remark}
\label{what goes wrong}
If we consider subdivided populations
(i.e. an analogous model on a lattice),
 then the analogous system of 
stochastic (ordinary) differential equations satisfies a duality of this
form with a system of branching and annihilating random walks.
Moreover, we can apply the analogous transformation to the stochastic 
p.d.e.~with coloured noise, but now, when we seek a branching annihilating dual 
process, in addition to the branching annihilating term when dual
particles meet, we obtain cross terms of the form
$$g(x,y)X_t(x)X_t(y)\Big(1-X_t(x)^2\Big)\Big(1-X_t(y)^2\Big).$$
We can rearrange the factors involving $X$ as 
\begin{multline}
\label{failed duality}
\left(X_t(x)^2X_t(y)^2-X_t(x)X_t(y)\right)+
\left(1-X_t(x)X_t(y)\right)
\\-
\left(X_t(x)^2-X_t(x)X_t(y)\right)
-\left(X_t(y)^2-X_t(x)X_t(y)\right).
\end{multline}
If $g(x,y)>0$, then we can interpret the first two terms 
in~(\ref{failed duality}) as branching and annihilating terms
in a putative dual; if
$g(x,y)<0$, then the last two terms can be interpreted as one
particle jumping to the location of another. However, we have not, in
either case, found a way to interpret both terms simultaneously.
The obvious approach
is to follow \cite{athreya/tribe:2000} and introduce an additional
`marker' that switches sign every time we have an event of 
`the wrong sign'. This leads to a 
Feynman-Kac correction term in the duality 
relation~(\ref{spatial duality}), which it turns out is infinite.
\end{remark}


\section{Tracer dynamics}
\label{tracers}

On passing to a stochastic p.d.e.~limit, we have lost sight of the way in
which individuals in the population were related to one another
and the ancestral selection graphs which encoded that information in
the prelimiting models do not converge. 
However, some information about heredity can be recovered using the notion
of `tracers'. The idea, which finds its roots in 
the statistical physics literature, has more recently found application in 
models of population genetics, notably in  \cite{hallatschek/nelson:2008},
or, for a more mathematical approach, see   \cite{durrett/fan:2016}.
The idea is simple: one labels some portion of the population of type 
$a$ individuals, say, at time zero not just according to their type at
the selected locus, but also with a `neutral marker' that is passed 
down from parent to offspring. Individuals in the population at time
$t$ that carry the neutral marker are precisely the 
descendants of our original marked individuals.

To introduce this in our setting, let us label a portion of the 
$a$ population at time zero by a neutral marker. We shall use $v(t,x)$
to denote the proportion of the (total) population that are both type
$a$ and labelled, and we shall use $a^*$ to denote that 
combined type. Thus
$$\{\mbox{type }a^*\mbox{ individuals}\}\subseteq  
\{\mbox{type }a\mbox{ individuals}\}$$
and $v(t,x)\leq w(t,x)$. 

The dynamics are driven by the same Poisson point processes of events as
before, but now we modify our description of inheritance to include
the extra label.
\begin{definition}[The SLFVFS with tracers]
Let $\zeta$, $\Pi^{neu}$, $\Pi^{fsel}, \Pi^{env}$ be exactly as in 
Definition~\ref{SLFVSFE_def}.
The dynamics of the pair $(v,w)$ can be described as follows.
Write $a^\dagger$ to denote individuals of type $a$, but not $a^*$.
\begin{enumerate}
\item If $(t,x,r,\bar{u}) \in \Pi^{neu}$, a neutral event occurs at 
time $t$ within the closed ball $B(x,r)$. Then:
\begin{enumerate}
\item Choose a parental location $l$ according to 
the uniform distribution over $B(x,r)$. 
\item Choose the parental type $\kappa \in \{0,1\}$ according to distribution
\begin{align*}
\mathbb{P}\left[\kappa = a^{*} \right] &= v (t_{-},l),
\quad \mathbb{P}\left[\kappa = a^{\dagger} \right] = w (t_{-},l) - v (t_{-},l),
\\
\mathbb{P}\left[\kappa = A \right] &= 1 -  w (t_{-},l).
\end{align*}
\item For each $y \in B(x,r)$, 
\begin{align*}
v(t,y) &= (1-\bar{u})v(t_{-},y) + \bar{u}\ind_{\{\kappa = a^{*}\}}(y);\\
w(t,y) &= (1-\bar{u})w(t_{-},y) + 
\bar{u}\ind_{\{\kappa \in\{ a^{*}, a^{\dagger}\}\}}(y).
\end{align*}
\end{enumerate}
\item If $(t,x,r,\bar{u}) \in \Pi^{fsel}$, a   selective event  occurs at time $t$ within the closed ball $B(x,r)$. Then:
\begin{enumerate}
\item Choose the two parental locations $l_{0}, l_{1}$ independently, 
according to the uniform distribution on $B(x,r)$.
\item Choose the two parental types, $\kappa_{0}, \kappa_{1}, $ according to  
\begin{align*}
\mathbb{P}\left[\kappa_{i} = a^{*} \right] &= v (t_{-},l_{i}),
\quad \mathbb{P}\left[\kappa_i = a^{\dagger} \right] = 
w (t_{-},l_{0}) - v (t_{-},l_{0}),\\
\mathbb{P}\left[\kappa_{i} = A \right] 
&= 1 -  w (t_{-},l_{i}).
\end{align*}
\item 
\begin{enumerate}
\item If  $\zeta(t,x)=1$, offspring inherit type $\kappa_0$ if $\kappa_0$ and
$\kappa_1$ are both type $a$ (with or without the neutral marker), 
otherwise they are type $A$; so
\begin{align*}
v(t,y) & = (1 - \bar{u})v(t_{-},y)  + \bar{u}\ind_{\{\kappa_{0} = a^*, 
\kappa_{1}\in\{a^*, a{\dagger}\}\}}(y) \\
w(t,y) & = (1 - \bar{u})w(t_{-},y)  + \bar{u}
\ind_{\{\kappa_{0},\kappa_{1}\in\{a^*, a^\dagger\}\}}.
\end{align*}
\item If $\zeta(t,x) =-1$, offspring inherit type $\kappa_0$ if $\kappa_0$ is
type $a$, and they inherit type $\kappa_1$ if $\kappa_0$ is type $A$.
Thus
\begin{align*}
v(t,y) & = (1 - \bar{u})v(t_{-},y)  + \bar{u}\left( 
\ind_{\{\kappa_{0} = a^*\}}+\ind_{\{\kappa_0=A, \kappa_{1} = a^*\}}\right);\\
w(t,y) & = (1 - \bar{u})w(t_{-},y)  + \bar{u}
\ind_{\{\kappa_{0} = \kappa_{1} = A\}^c}.
\end{align*} 
\end{enumerate}
\end{enumerate}
\end{enumerate}
\end{definition}

\begin{theorem}
\label{tracer theorem}
Applying our previous scalings, 
let the population evolve under the assumptions of Theorem~\ref{scaling result}.

Assume that the sequence of measures of the initial states of the marked population
converges weakly in $\mathcal{M}_{\lambda}$ to a measure with a density
 $v^\infty(0,x)= \lim_{n \to \infty} \overline{v}_{n}(0,x)$.
Then as $n\rightarrow\infty$,
the corresponding sequence of rescaled measure-valued processes is tight 
in $D([0,\infty),\mathcal{M}_{\lambda}^{2})$,
and any limit point is a weak solution to a system of stochastic
p.d.e.'s which in $d\geq 2$ takes the form
\begin{align*}
\mathrm{d}w^{\infty} \!&=\!  \left(\frac{\bar{u}\Gamma_R}{2}
\Delta w^{\infty} + V_R^2\bar{u}^{2}{\vs}^{2}w^{\infty}(1 - w^{\infty})
(1 - 2w^{\infty})\right)\!\mathrm{d}t  \!
+\!  \sqrt{2}V_R\bar{u}\vs w^{\infty}(1 - w^{\infty}){W}(\mathrm{d}t,\mathrm{d}x),
 \\
dv^{\infty} &= \left( \frac{\bar{u}\Gamma_R}{2}
\Delta v^{\infty} + V_R^2\bar{u}^{2}{\vs}^{2}v^{\infty}(1 - w^{\infty})
(1 - 2w^{\infty})\right)\!\mathrm{d}t \!
+\!\sqrt{2}V_R\bar{u}\vs v^{\infty}(1 - w^{\infty}){W}(\mathrm{d}t,\mathrm{d}x),
\end{align*}
where the noise ${W}$ is as before.

For dimension $d=1$, the limiting process is a weak
solution to the system of stochastic partial differential equations 
\begin{align*}
\mathrm{d}w^{\infty} \!=&  \left(\frac{\bar{u}\Gamma_R}{2}
\Delta w^{\infty} + 4R^2\bar{u}^{2}{\vs}^{2}w^{\infty}(1 - w^{\infty})
(1 - 2w^{\infty})\right)\mathrm{d}t
 +   \sqrt{2}2R\bar{u}w^{\infty}(1 - w^{\infty}){W}(\mathrm{d}t,\mathrm{d}x)\\ 
&+ 2R\bar{u}\sqrt{v^{\infty}(1 - w^{\infty})}{\mathcal{W}^0}(\mathrm{d}t,\mathrm{d}x) 
+ 2R\bar{u}\sqrt{(w-v)^{\infty}(1 - w^{\infty})}{\mathcal{W}^1}(\mathrm{d}t,\mathrm{d}x), \\
\mathrm{d}v^{\infty} \! =& \left(\frac{\bar{u}\Gamma_R}{2}\Delta 
v^{\infty} + 4R^2\bar{u}^{2}{\vs}^{2}v^{\infty}(1 - w^{\infty})
(1 - 2w^{\infty})\right)\mathrm{d}t 
+  \sqrt{2}2R\bar{u}{\vs} v^{\infty}(1 - w^{\infty})
{W}(\mathrm{d}t,\mathrm{d}x)\\
& + 
2R\bar{u}\sqrt{v^{\infty}(1 - w^{\infty})}{\mathcal{W}^0}(\mathrm{d}t,\mathrm{d}x)
+2R\bar{u}\sqrt{v^{\infty}(w - v^{\infty})}{\mathcal{W}^2}(\mathrm{d}t,\mathrm{d}x),
\end{align*}
where, as before, ${W}$ is white in time and coloured in space and 
and ${\mathcal{W}^i}$, $i=0,1,2$, are independent space-time white noises.
\end{theorem}

The proof of this result is an even longer version of the proof of
Theorem~\ref{scaling result},
but it follows exactly the same strategy. First we show that the limit
points are solutions to appropriate martingale problems, then we show that
any solution to the martingale problem provides a weak solution to the
system of stochastic p.d.e.'s. 
Rather than providing details of the proof, we indicate 
why this result is to be expected. We once again use the trick
of \cite{kurtz:1973}. In the case of one dimension and 
genic selection,  \cite{durrett/fan:2016}
obtain a pair of stochastic p.d.e.'s of the form
\begin{eqnarray*}
\mathrm{d}w &=&\left(\alpha\Delta w+\vs w(1-w)\right)\mathrm{d}t
+\sqrt{v(1-w)}W^0(dt,dx)
+\sqrt{(w-v)(1-w)}W^{1}(\mathrm{d}t,\mathrm{d}x)\\
\mathrm{d}v &=&\left(\alpha\Delta v +\vs v(1-w)\right)\mathrm{d}t+\sqrt{v(1-w)}W^0(\mathrm{d}t,\mathrm{d}x)
+\sqrt{v(w-v)}W^{2}(dt,dx),
\end{eqnarray*}
where $W^0$, $W^1$ and $W^2$ are independent space-time white noises.
Writing the corresponding generator acting on test functions $f(w,l)$
as ${\cal L}^{neu}+ {\cal L}^{sel}$ as in Section~\ref{nospace},
to identify the generator in the limit of (appropriately
scaled) rapidly fluctuating selection we must evaluate 
${\cal L}^{sel}({\cal L}^{sel} f(\cdot, \cdot))(w,v)$ which, up to constants, is
\begin{multline*}
\left(w(1-w)\frac{\partial}{\partial w}+ 
v(1-w)\frac{\partial}{\partial v}\right)
\left\{w(1-w)\frac{\partial f}{\partial w}+ v(1-w)\frac{\partial f}{\partial v}
\right\} 
\\ =
w(1-w)(1-2w)\frac{\partial f}{\partial w}
+
v(1-w)(1-2w)\frac{\partial f}{\partial v}
\\+
w^2(1-w)^2\frac{\partial^2 f}{\partial w^2}
+2 w v (1-w)^2\frac{\partial^2 f}{\partial v\partial w}
+v^2(1-w)^2\frac{\partial^2 f}{\partial v^2}.
\end{multline*}
From this we see that the stochastic p.d.e's in Theorem~\ref{tracer theorem}
are of precisely the form that we should expect.

\section{Numerical results}
\label{numerics}

In order to gain a little more intuition about the effects of 
fluctuating selection on allele frequencies, in this section we 
present the results of a simple numerical experiment. It is certainly
not an exhaustive study, but it points to some of the challenges that face
us in distinguishing causes of patterns of allele frequencies. Our simulations are not of the spatial Lambda-Fleming-Viot models, but 
of natural extensions of the classical Moran model to incorporate spatial structure and fluctuating selection. 
After suitable scaling, we expect allele frequencies under these models to
converge to the same limiting stochastic p.d.e.~as our
scaled SLFVFS.

\begin{definition}[Spatial Moran model with fluctuating selection]
\label{SMMFS}
The population, which consists of two genetic types, $\{a,A\}$,
lives at the vertices of a 
discrete lattice $\mathbb L$. There are $N_d$ individuals at each vertex
(or {\em deme}). The state of the environment at time $t$ in deme $x$
is denoted by $\zeta(t,x)$. 

The dynamics of the process are described as follows:
\begin{enumerate}
\item{\bf Reproduction events}
{\begin{enumerate}
\item{\bf Neutral events:} For each deme, independently, at rate
$\binom{N_d}{2}$ a pair of individuals is chosen (uniformly at random), one of
the pair (picked at random) dies and the other splits in two;
\item{\bf Selective events:} For each deme $x$, independently, at rate
$\vs N_d$ a pair of individuals is chosen (uniformly at random), one
individual splits in two and the other one dies,  
if $\zeta(t,x)=-1$ and at least one of the pair is type $a$, then it is 
a type $a$ individual that is chosen to split, whereas if 
$\zeta(t,x)=1$ and at least one of the pair is type $A$, then a type $A$
individual is chosen to split.
\end{enumerate}}
\item{\bf Migration events:} For each pair of demes $x_1$, $x_2$, we associate
a nonnegative parameter $m_{x_1,x_2}$. Independently for each pair, 
at rate $m_{x_1,x_2}N_d$, an individual
is chosen uniformly at random from each of the demes $x_1$, $x_2$ and they
exchange places.
\item{\bf Environmental events:} At the times of a Poisson process 
of rate $\alpha$, which is 
independent of those driving reproduction and migration, the environment
is resampled. The value of the environment variable at each deme is 
uniformly distributed on $\{-1,+1\}$ and $\IE[\zeta(t,x)\zeta(t,y)]=g(x,y)$ for 
a correlation function $g$.
\end{enumerate}     
\end{definition}


In the experiments that follow, 
we take the lattice $\mathbb L$ to be a circle of $100$ demes with nearest
neighbour migration at rate $1$. We set $N_d=400$, $\vs =0.1$ and 
$\alpha=10$. 
The environment variables in demes $0-50$ all take the same value, 
as do those in demes $51-100$.  
We consider four different scenarios: 
\begin{enumerate}
\item Demes $0-50$ and $51-100$ are perfectly anticorrelated; 
the environment fluctuates in time and the direction of selection is 
always different in demes  $0-50$ and $51-100$;
\item Demes $0-100$ are perfectly correlated; 
the direction of selection fluctuates in time but it is the same in every deme;
\item Constant selection (the environment is fixed), with the 
direction of selection in demes  $0-50$ the opposite of that in demes  $51-100$;
\item The neutral case.
\end{enumerate}

To ensure comparability of experiments, the same events are used for all
the scenarios, with the only difference lying in the value of the 
environment variable. Thus, in the neutral case, either individual
is equally likely to be the parent in `selective events', irrespective of 
type. 
A more precise description of the code used for simulations can be found in 
Appendix~\ref{code}.


As a first comparison, Figure~\ref{global proportions} shows the proportion
of type $a$ individuals across the whole population. This is just a 
single realisation of the experiment. There is certainly
no dramatic divergence from neutrality. However, as we illustrate 
below, this can mask some more interesting effects at the local level.
\begin{figure}[h]
\includegraphics[width=4in]{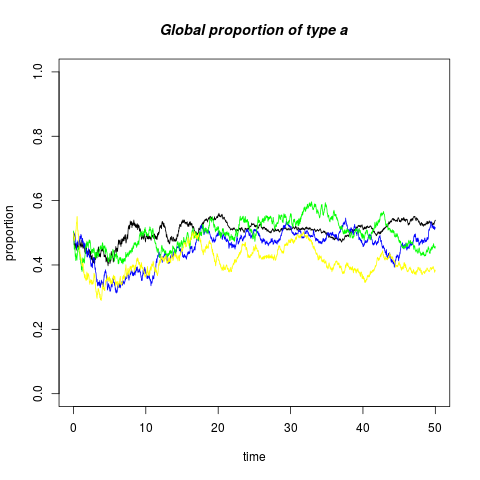}
\caption{Global proportions of type $a$ individuals; one realisation, four
different scenarios: blue anticorrelated environments, green correlated 
environments, black constant selection, and yellow neutral. See the main 
text for a full explanation.} 
\label{global proportions}
\end{figure}

In Figure~\ref{proportions in demes}, we have used a greyscale to record the
proportion of type $a$ individuals in each deme - the darker the colour,
the greater the proportion of type $a$. In the top left, selection is fixed, 
and we clearly see the effect of type $a$ being favoured in demes $50-100$.
In the next two frames (top right and bottom left), the environment 
fluctuates, but whereas on the top left demes $0-50$ always 
favour the opposite type to demes $51-100$, on the bottom left all demes 
always favour the same type. The neutral model is the bottom right.
When we repeat over many realisations, we see a
greater concentration of types than for the neutral model, but
it is certainly not easy to distinguish between the two frames.

\begin{figure}
    \centering
    \begin{subfigure}[b]{0.48\textwidth}
        \includegraphics[width=\textwidth]{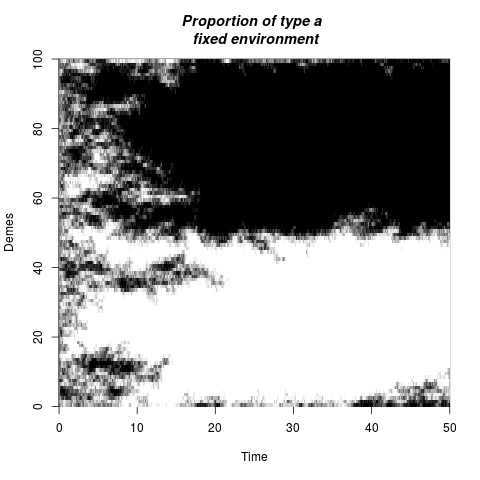} 
    \end{subfigure}
    ~ 
    \begin{subfigure}[b]{0.48\textwidth}
        \includegraphics[width=\textwidth]{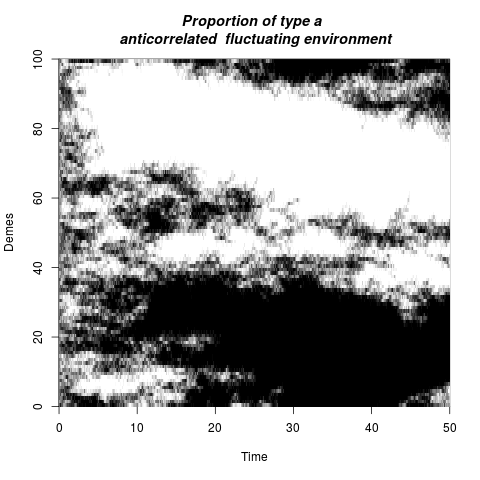}
    \end{subfigure}
    ~ 
    
        \centering
    \begin{subfigure}[b]{0.48\textwidth}
        \includegraphics[width=\textwidth]{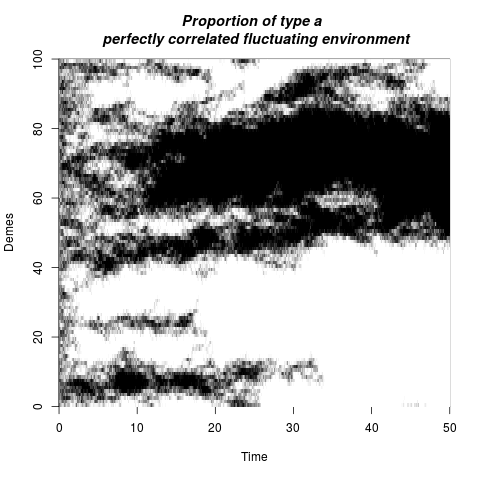} 
    \end{subfigure}
    ~ 
    \begin{subfigure}[b]{0.48\textwidth}
        \includegraphics[width=\textwidth]{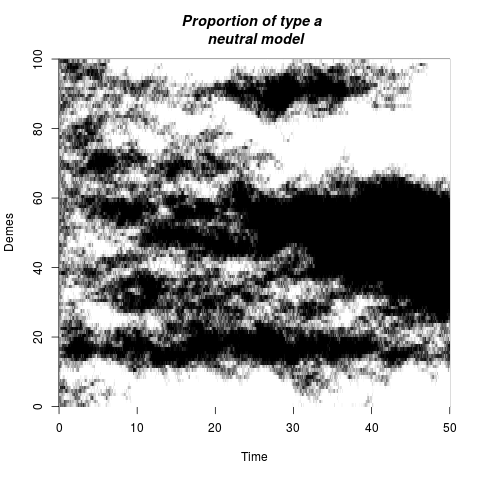}
    \end{subfigure}
    \caption{Map of the proportion of type $a$ individuals in each deme;
the darker the colour the higher the proportion of type $a$. Top left, constant
selection; top right, fluctuating selection with demes $0-50$ and $51-100$
perfectly anticorrelated; bottom left, fluctuating selection with all demes 
perfectly correlated; bottom right, the neutral case.}
\label{proportions in demes}
\end{figure}

Just following 
the overall proportion of types
in a deme is throwing away a lot of information, which
may be available in genetic data, about the distribution of families. 
To explore this we used `tracers', further explored in the context of 
Spatial Lambda-Fleming-Viot process in  Section~\ref{tracers}. In 
Figure~\ref{tracers fixed envt}
we mark individuals descended from the population in particular demes at
time zero. The darker the colour, the higher the proportion of marked
individuals. In a constant environment, left panel, the surviving family is well
adapted to the environment in demes $50-100$, but has difficulty invading 
demes $1-50$, where it is not favoured. It also struggles to expand beyond 
deme $80$. This turns out to be because of competition with the equally 
well adapted family that is descended from individuals that were in deme $84$
at time zero (right hand panel). 

\begin{figure}
    \centering
    \begin{subfigure}[b]{0.48\textwidth}
        \includegraphics[width=\textwidth]{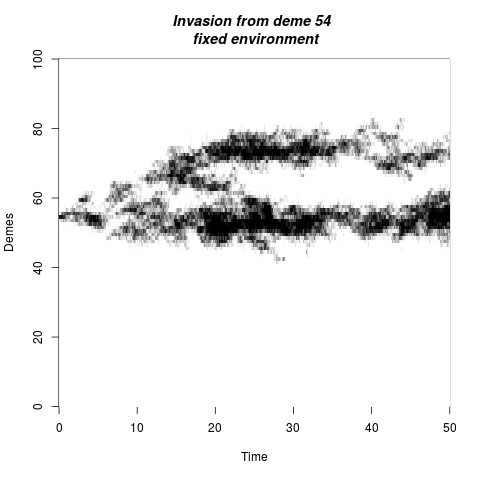} 
    \end{subfigure}
    ~ 
    \begin{subfigure}[b]{0.48\textwidth}
        \includegraphics[width=\textwidth]{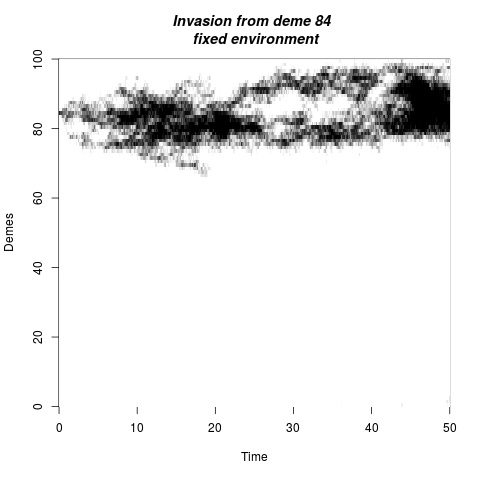}
    \end{subfigure}
    ~ 
    \caption{Tracer map of individuals descended from ancestors located in 
particular demes at
time zero; the darker the colour, the higher the proportion of such 
individuals. Fixed environment.}
\label{tracers fixed envt}
\end{figure}

In left panel of Figure~\ref{tracers: fluctuating envt}, 
we see a successful family in
a fluctuating environment (with demes $0-50$ and $51-100$ perfectly
anticorrelated). The family is on the brink of extinction several times and
is rescued by a change in the environment. In the right panel, we 
see a family that begins life right on the boundary between the two regions.
The environments are perfectly anticorrelated and the family is able to
survive and spread much more readily.

\begin{figure}
    \centering
    \begin{subfigure}[b]{0.48\textwidth}
        \includegraphics[width=\textwidth]{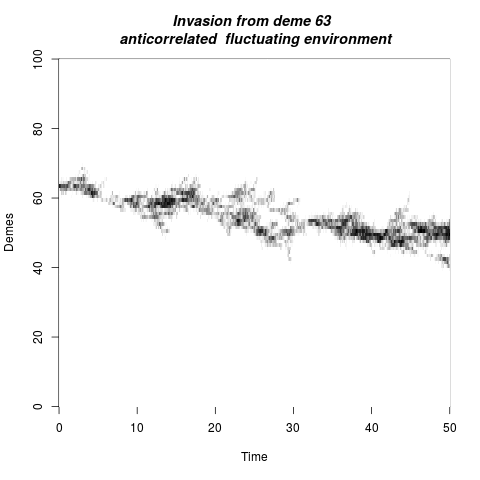} 
    \end{subfigure}
    ~ 
    \begin{subfigure}[b]{0.48\textwidth}
        \includegraphics[width=\textwidth]{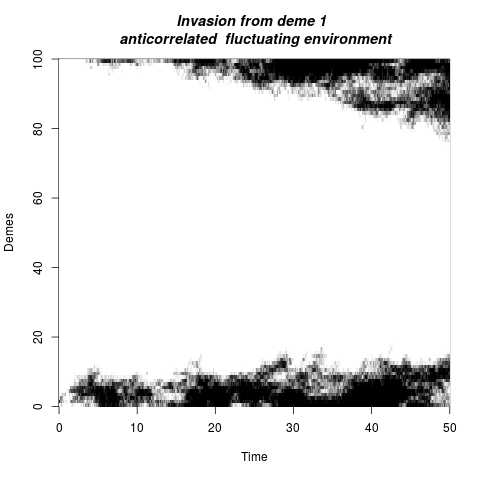}
    \end{subfigure}
    ~ 
    \caption{Tracer map of individuals descended from ancestors located in 
particular demes at
time zero; the darker the colour, the higher the proportion of such 
individuals. The environment is fluctuating, with its value in 
demes $0-50$ perfectly anticorrelated with that in demes $51-100$.}
\label{tracers: fluctuating envt}
\end{figure}
The right hand panel of Figure~\ref{tracers: fluctuating envt} 
can be contrasted with
the left hand panel of Figure~\ref{tracers: fluctuating correlated envt}. The 
descendants of ancestors in deme $1$ find it harder to spread in a perfectly
correlated environment than in the perfectly anticorrelated environment of 
the previous figure. The trace of descendants of ancestors in deme $78$ in
the right hand panel of Figure~\ref{tracers: fluctuating correlated envt} 
shows the
`thinning' of the family resulting from the periods of time when it is not
favoured. 
\begin{figure}
    \centering
    \begin{subfigure}[b]{0.48\textwidth}
        \includegraphics[width=\textwidth]{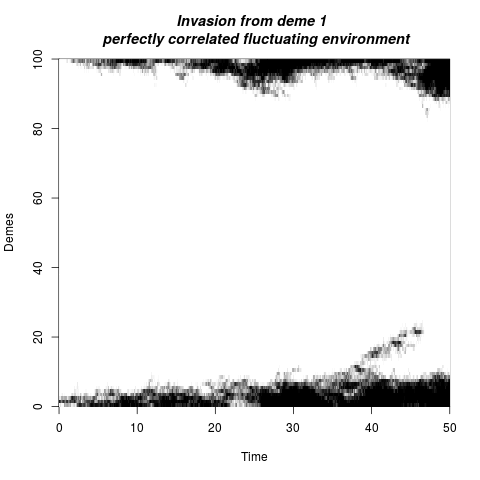} 
    \end{subfigure}
    ~ 
    \begin{subfigure}[b]{0.48\textwidth}
        \includegraphics[width=\textwidth]{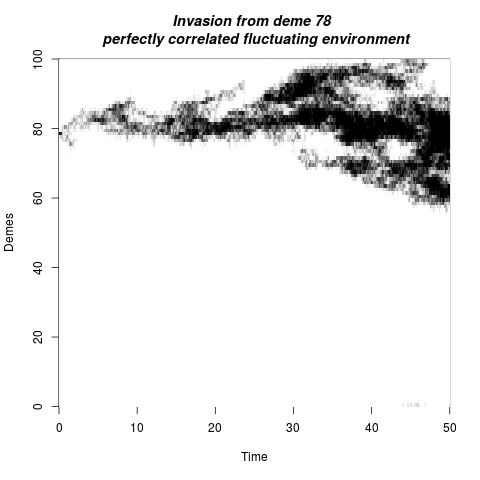}
    \end{subfigure}
    ~ 
    \caption{Tracer map of individuals descended from ancestors located in 
particular demes at
time zero; the darker the colour, the higher the proportion of such 
individuals. The environment is fluctuating, with its value in 
demes $0-50$ perfectly correlated with that in demes $51-100$.}
\label{tracers: fluctuating correlated envt}
\end{figure}

Finally, Figure~\ref{tracers: neutral}, shows the trace of descendants of 
ancestors
in demes $12$ and $16$ for the neutral model. There appears to be a barrier
between the two families, which could easily be mistaken for a change in 
the environment somewhere between demes $12$ and $16$, on the lower side
of which the family descended from deme $12$ is better adapted and on the
upper side of which the family from deme $16$ is better adapted. In fact this
is due to competition between two equally fit families.  

\begin{figure}
    \centering
    \begin{subfigure}[b]{0.48\textwidth}
        \includegraphics[width=\textwidth]{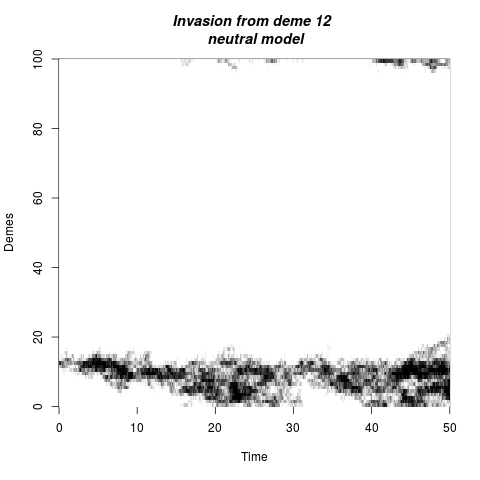} 
    \end{subfigure}
    ~ 
    \begin{subfigure}[b]{0.48\textwidth}
        \includegraphics[width=\textwidth]{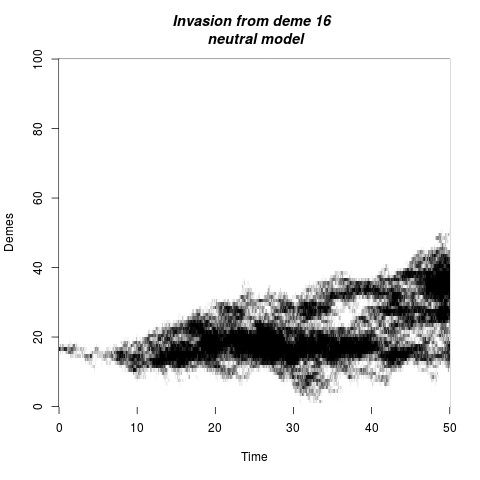}
    \end{subfigure}
    ~ 
    \caption{Tracer map of individuals descended from ancestors located in 
particular demes at
time zero; the darker the colour, the higher the proportion of such 
individuals. Neutral model.}
\label{tracers: neutral}
\end{figure}

\section{Proof of Theorem~\ref{scaling result}}\label{scaling proof}

The proof of Theorem~\ref{scaling result} will rest on Theorem~2.1 of
 \cite{kurtz:1992} (or rather his Example~2.2). 
For a metric space $E$, let 
$l_m(E)$
be the space of measures on $[0,\infty) \times E$ such that 
$\mu \in l_m(E)$ if $\mu([0,t) \times E) = t$.
\begin{theorem}[ \cite{kurtz:1992}, Theorem 2.1]
\label{kurtzaveraging}
Let $E_1$, $E_2$ be complete separable metric spaces, 
and set $E=E_1\times E_2$. For each $n$, let $\{(X_n,Y_n)\}$ be 
a stochastic process with sample paths in $D_E([0,\infty))$ adapted to 
a filtration $\{{\cal F}_t^n\}$. Assume that 
$\{X_n\}$ satisfies the compact containment condition, that is, for 
each $\epsilon>0$ and $T>0$, there exists a compact $K\subset E_{1}$ such
that
\begin{equation}
\label{tightness}
\inf_n\IP[X_n(t)\in K, t\leq T]\geq 1-\epsilon,
\end{equation}
and assume that $\{Y_n(t) : t\geq 0, n=1,2,\ldots\}$ is 
relatively compact (as a collection of $E_2$-valued random 
variables). Suppose that there is an operator 
$A:{\cal D}(A)\subset \overline{C}(E_1)\rightarrow C(E_1\times E_2)$ such
that for $f\in {\cal D}(A)$ there is a process $\epsilon_n^f$ for 
which 
\begin{equation}
\label{kurtz rearrangement}
f(X_n(t))-\int_0^t Af(X_n(s),Y_n(s))\mathrm{d}s +\epsilon_n^f(t)
\end{equation}
is an $\{{\cal F}_t^n\}$-martingale. Let ${\cal D}(A)$ be dense
in $\overline{C}(E_1)$ in the topology of uniform 
convergence on compact sets. Suppose that for each $f\in{\cal D}(A)$ and
each $T>0$, there exists $p>1$ such that
\begin{equation}
\label{kurtzbound}
\sup_n\IE\left[\int_0^T|Af(X_n(t),Y_n(t)|^{p}\mathrm{d}t\right]<\infty
\end{equation}
and 
\begin{equation}
\label{kurtzboundepsilon}
\lim_{n\rightarrow\infty}\IE\left[\sup_{t\leq T}|\epsilon_n^f(t)|\right]=0.
\end{equation}
Let $\Gamma_n$ be the $l_m(E_2)$-valued random variable given by
\begin{equation}
\nonumber
\Gamma_n\left([0,t]\times B\right)=\int_0^t\ind_B(Y_n(s))\mathrm{d}s.
\end{equation}
Then $\{(X_n,\Gamma_n)\}$ is relatively compact in 
$D_{E_1}[0,\infty)\times l_m(E_2)$, and for any limit point
$(X,\Gamma)$ there exists a filtration $\{{\cal G}_t\}$ such that
\begin{equation}
f(X(t))-\int_0^t\int_{E_2}Af(X(s),y)\Gamma(\mathrm{d}s\times \mathrm{d}y)
\end{equation}
is a $\{{\cal G}_t\}$-martingale for each $f\in {\cal D}(A)$.
\end{theorem}
As a particular case,  \cite{kurtz:1992} provides the following example.
\begin{example}[ \cite{kurtz:1992}, Example 2.2]
\label{kurtzexample}
Suppose $Y$ is stationary and ergodic, and $Y_n$ in 
Theorem~\ref{kurtzaveraging}
is given by $Y_n(t)\equiv Y(nt)$. Then $\Gamma=m\times \pi$ where 
$m$ denotes Lebesgue measure and $\pi$ is the marginal distribution 
for $Y$. Consequently, under the assumptions of 
Theorem~\ref{kurtzaveraging},
$X$ is a solution of the martingale problem for $C$ given by
$$Cf(x)=\int Af(x,y)\pi(\mathrm{d}y),\qquad f\in {\cal D}(A).$$
\end{example}

Let us briefly discuss how our setup fits into 
that of Theorem~\ref{kurtzaveraging}. 
The sequence of measure-valued evolutions 
$\big(\overline{M}^{n}_{t} \big)_{t \geq 0}$ corresponds to the 
process $X_{n}$, while the sequence of environments corresponds to the 
process $Y_{n}$.
Notice that $Y_{n}$  defined that way satisfies the assumptions of 
Example~\ref{kurtzexample}, and therefore those of Theorem~\ref{kurtzaveraging}.
The domain of the generator $A$ is given by the closure of
\begin{align}
\label{domain of generator}
\mathcal{D}(A)  = \bigg\lbrace  
F(\langle \cdot, w \rangle) : f \in C_{c}^{\infty}(\mathbb{R}^{d}), 
F \in C^{\infty}(\mathbb{R})
\bigg\rbrace.
\end{align}
Since the state space of  $\big(\overline{M}^{n}_{t} \big)_{t \geq 0}$ is compact, the compact containment condition \eqref{tightness} is satisfied.

The key step in applying Theorem~\ref{kurtzaveraging} is the identification
of the generator $A$, which leads to the 
decomposition~\eqref{kurtz rearrangement}.
This consists of two steps. First we approximate the part of the 
SLFVFS generator corresponding to neutral events, and second we deal 
with the terms corresponding to fluctuating selection. 
These two steps are accomplished in 
Subsection~\ref{Subsection Identifyiing the limit}.

Our approach mirrors the strategy of Section~\ref{nospace}.
Having transformed the generator of the SLFVFS into a form which 
allows us to consider the process of averaged densities 
$\overline{w}^{n}_{t}$, we use the `separation of timescales' trick 
of  \cite{kurtz:1973}, to (up to a small error)
split the generator into parts corresponding to neutral and selective 
events in a convenient way; see equation~\eqref{Limit formal evaluation step 1}.
We then consider the two terms in this decomposition.
The analysis of the neutral part of the generator 
will closely follow 
\cite{etheridge/veber/yu:2014}; the novel step is the second, 
which deals with the
terms corresponding to fluctuating selection.
Our analysis will not only identify the correct form of $A$ and 
$\epsilon_n^{f}$ (see \eqref{form_of_A} and \eqref{epsilonnf} respectively), 
but also provides a uniform bound on $A$, which shows 
that~\eqref{kurtzbound} is satisfied.

Having identified $A$, and $\epsilon_n^{f}$, we show 
that~\eqref{kurtzboundepsilon} holds.
This bound relies on the Lipschitz bound~\eqref{regularity of g} on 
the correlation kernel $g(x,y)$.
Subsection~\ref{Subsection Identifyiing the limit} concludes with a 
discussion of the continuity estimates on the averaged densities of the
measure-valued evolutions that are required to deduce our result. 
These estimates are stated as 
Equation~\eqref{ctty eq}
and Proposition~\ref{continuity proposition} (and
proved in Appendix~\ref{continuity estimates}).
This leads to a complete characterisation of limits
of the sequence of averaged processes $\overline{w}_n$ as solutions
to a martingale problem.

%
%
%


In Subsection~\ref{Subsection Uniqueness of solutions for limiting equation} 
we discuss the 
uniqueness of solution to the limiting martingale problem, which we have only managed to prove 
in dimensions $d\geq 2$.
This requires a separate argument
based on the fact that, under certain regularity conditions 
(which are satisfied as a result of 
Proposition~\ref{continuity proposition}), 
any solution to the martingale 
problem is a weak solution to a corresponding stochastic p.d.e.
The weak uniqueness of solutions to the stochastic 
p.d.e.~is a consequence of 
a combination of a pathwise uniqueness result of 
 \cite{rippl/sturm:2013} and 
a Yamada-Watanabe Theorem.

\subsection{Identifying the limit}
\label{Subsection Identifyiing the limit}
In what follows $C$ is a constant that depends on certain fixed quantities 
such as the functions $f$ and $F$ that appear in our test functions,
c.f.~\eqref{domain of generator}, and the quantities
$\vs$ and $\bar{u}$ appearing in our scaled selection and impact parameters
in the statement of Theorem~\ref{scaling result}.
The quantity $C$ may
vary from line to line; its
exact value is unimportant.

Throughout this section we shall be manipulating expressions pertaining 
to a local average of a version of the density of the SLFVFS. These
manipulations are insensitive to changing that density on a set of
Lebesgue measure zero and so should really be thought of as 
results for the measure-valued evolution itself.

Our result concerns the locally averaged process $\overline{w}_n$. However, 
before passage to the limit, this process is not Markov (in contrast
to $w$ itself). In order to overcome this, we follow
\cite{etheridge/veber/yu:2014} and define 
 \begin{align*}
 \phi_{f}(x) =  \vint_{B_{n}(x)}f(y) \mathrm{d}y.
 \end{align*}
To alleviate notation, we have suppressed the dependence on $n$ in 
this quantity. 
Notice that
\begin{align*}
\langle w_{n},\phi_{f}\rangle = \int_{\mathbb{R}^{d}}
\int_{\mathbb{R}^{d}}w_{n}(y) \frac{n^{\frac d3}}{V_{R}} f(z) \ind_{\{|z-y| \leq R_{n}\}}(z)\mathrm{d}y \mathrm{d}z = \langle \overline{w}_{n},f\rangle.
\end{align*}
Using Taylor's Theorem, we see that for $f\in C^2_c$, 
\begin{align}\label{taylor of phi}
\phi_{f}(x) &= f(x) + \frac{n^{\frac d3}}{V_{R}}  
\int_{B_{n}(x)}(y-x)\cdot \nabla f(x) \mathrm{d}y + 
\mathcal{O}\left( \frac{n^{\frac d3}}{V_{R}} 
\int_{B_{n}(x)\cap S_{f}}|x-y|^{2} \mathrm{d}x\right) \nonumber\\
& = f(x) + \mathcal{O}\left(n^{-\frac23} \right)\ind_{B_{n}(x)\cap S_{f}\neq
\emptyset},
\end{align}
where $S_{f}$ denotes the support of $f$. Here, and throughout, we have used 
$|\cdot|$ to denote the Euclidean norm.
This yields the useful approximation
\begin{align} \label{<w,phif> = <barw,f>}
\langle \overline{w}_{n},f \rangle = \langle w_{n},\phi_{f} \rangle = \langle w_{n},f \rangle + \mathcal{O}(n^{-\frac{2}{3}}).
\end{align}

It will be convenient to have some notation.  
\begin{notn}\label{Remark change def}
Suppose that immediately before a reproduction 
event the unscaled process takes the value $w$,
then immediately after the event, its state will be 
$\theta^+_{x,R,\bar{u}}(w)$, if
the parent of the event was of type $a$, and 
$\theta^{-}_{x,R,\bar{u}}(w)$ if the parent
was of type $A$, where
\begin{align}\label{change def}
\theta^{+}_{x,R,\bar{u}}(w) &= \ind_{\mathbb{R}^{d} \setminus B(x,R)}w + 
\ind_{B(x,R)}\left((1 - \bar{u})w + \bar{u}\right), \nonumber\\ 
\theta^{-}_{x,R,\bar{u}}(w) &= \ind_{\mathbb{R}^{d} \setminus B(x,R)}w + 
\ind_{B(x,R)}(1 - \bar{u})w.
\end{align}
\end{notn}
Using this notation, 
\begin{eqnarray}
\big\langle \theta^{+}_{x,R_{n},u_{n}}(w), \phi_f \big\rangle 
- \langle w,\phi_f \rangle 
= u_{n}\left\langle \ind_{B_{n}(x)}(1 - w),\phi_f\right\rangle,\label{theta+ w}
\\
\label{theta- w}
\big\langle \theta^{-}_{x,R_{n},u_{n}}(w), \phi_f \big\rangle - 
\left\langle w,\phi_f \right\rangle = -u_{n}\left\langle \ind_{B_{n}(x)}w,\phi_f\right\rangle.
\end{eqnarray}

The first step will be to evaluate the generator on test functions of the
form 
\begin{equation}
\label{test functions}
\psi_{F,f}(w_n)=F\left(\langle \overline{w}_n,f\rangle\right)
=F\left(\langle w_n,\phi_f\rangle\right),
\end{equation}
where $F\in C^\infty(\IR)$, $f\in C^\infty_c(\IR^d)$. Notice that these are
independent of the environment $\zeta$.
Recall that the analogous functions
were our starting point in Section~\ref{nospace}.
Since the support $S_f$ of $f$ is compact, it is only overlapped by
events at a finite rate and so, writing ${\cal L}\psi_{F,f}$,
which will be a function of the state $\zeta(\cdot)$ of the environment, 
for the
generator of the unscaled process
acting on test functions of this form, we have
\begin{eqnarray*}
\mathcal{L}\psi_{F,f}(w,\zeta)&=&  
(1-{\vs})\int_{\mathbb{R}^{d}}\vint_{B(x,R)}w(y)
\left[ F(\langle \theta^{+}_{x,R,\bar{u}}(w),\phi_f \rangle) 
- F(\langle w,\phi_f \rangle) \right]\mathrm{d}y\mathrm{d}x\\ 
&&+(1-\vs)\int_{\mathbb{R}^{d}}\vint_{B(x,R)}\lbrace 1-w(y)\rbrace
\left[ F(\langle \theta^{-}_{x,R,\bar{u}}(w),\phi_f \rangle) 
- F(\langle w,\phi_f \rangle) \right]\mathrm{d}y\mathrm{d}x \\
&&+ \vs\int_{\mathbb{R}^{d}}\vint_{B(x,R)^{2}}\Bigg\{ w(y_{1})w(y_{2}) 
+ \frac{1-\zeta(x)}{2}\Bigg[(w(y_{1})(1-w(y_{2})) \\
&& + (1-w(y_{2}))w(y_{2})\Bigg]\Bigg\}
\times\left[ F(\langle \theta^{+}_{x,R,\bar{u}}(w),\phi_f \rangle) 
- F(\langle w,\phi_f \rangle) \right]
\mathrm{d}y_{1}\mathrm{d}y_{2}\mathrm{d}x \\
&& +\vs\int_{\mathbb{R}^{d}}\vint_{B(x,R)^{2}}
\Bigg\{(1-w(y_{1}))(1-w(y_{2})) + \frac{1+\zeta(x)}{2}(w(y_{1})(1-w(y_{2})) \\
&& + (1-w(y_{1}))w(y_{2}) )\Bigg\}
\times\left[ F(\langle \theta^{-}_{x,R,\bar{u}}(w),\phi_f \rangle) 
- F(\langle w,\phi_f \rangle) \right]\mathrm{d}y_{1}\mathrm{d}y_{2}\mathrm{d}x.
\\
&:=&
\mathcal{L}^{neu}\psi_{F,f}(w)+ 
\mathcal{L}^{fsel}\psi_{F,f}(w,\zeta) \\ 
\end{eqnarray*}
We recognise the generator of the neutral spatial Lambda-Fleming-Viot process
(which does not depend on the environment)
plus two terms reflecting the selection events. If we set the environment to 
be identically equal to $1$, say, then the surviving selection term is precisely
that considered in  \cite{etheridge/veber/yu:2014} in modelling selection
in favour of type $A$ individuals.

After scaling (and a change of variables), 
the generator of $w_n$ acting on test functions
of this form is given by
\begin{align*}
\mathcal{L}_{n}\psi_{F,f}(w_n,\zeta)
&=n^{1 + \frac{d}{3}}(1-\vs_n)\left(\int_{\mathbb{R}^{d}}\overline{w}_n(x)
\left[ F(\langle \theta^{+}_{x,R_{n},u_{n}}(w_{n}),\phi_f \rangle) 
- F(\langle w_{n},\phi_f \rangle) \right]\mathrm{d}x\right.\\ 
  & \phantom{=n^{1 + \frac{d}{3}}} +\left.\int_{\mathbb{R}^{d}}\lbrace 1-\overline{w}_{n}(x)\rbrace 
\left[ F(\langle \theta^{-}_{x,R_{n},u_{n}}(w_{n}),\phi_f \rangle) 
- F(\langle w_{n},\phi_f \rangle) \right]\mathrm{d}x\right)\\
&+n^{1+\frac{d}{3}}
\vs_{n}\int_{\mathbb{R}^{d}}\left\lbrace\overline{w}_{n}^{2}(x) + 
(1-\zeta(x))\overline{w}_{n}(x)(1-\overline{w}_{n}(x))\right\rbrace\\
&\phantom{=n^{1 + \frac{d}{3}}=n^{1 + \frac{d}{3}}=n^{1 + \frac{d}{3}}\vs_{n}\int_{\mathbb{R}^{d}}}\times\left[ F(\langle \theta^{+}_{x,R_{n},u_{n}}(w_{n}),\phi \rangle) 
- F(\langle w_{n},\phi_f \rangle) \right]\mathrm{d}x\\
&+n^{1+\frac{d}{3}}
\vs_{n}\int_{\mathbb{R}^{d}}\left\lbrace(1-\overline{w}_{n}(x))^{2} +  
(1 + \zeta(x))\overline{w}_{n}(x)(1-\overline{w}(x)) \right\rbrace\\
&\phantom{=n^{1 + \frac{d}{3}}=n^{1 + \frac{d}{3}}=n^{1 + \frac{d}{3}}\vs_{n}\int_{\mathbb{R}^{d}}}\times\left[ F(\langle \theta^{-}_{x,R_{n},u_{n}}(w_{n}),\phi_f \rangle) 
- F(\langle w_{n},\phi_f \rangle) \right]\mathrm{d}x.
\end{align*}
By analogy with~\eqref{LFVres}, we write this as
\begin{align}\label{final gneator}
\mathcal{L}_{n}\psi_{F,f}
=(1-{\vs}_n)\mathcal{L}^{neu}_{n}\psi_{F,f} + 
n^{\alpha}\mathcal{L}^{fsel}_{n}\psi_{F,f},
\end{align}
where 
\begin{multline}
\label{neutral bit}
\mathcal{L}^{neu}_{n}\psi_{F,f}(w_n)=
n^{1 + \frac{d}{3}}\left(\int_{\mathbb{R}^{d}}
\overline{w}_{n}(x)\left[ F(\langle \theta^{+}_{x,R_{n},u_{n}}(w_{n}),
\phi_f \rangle) - F(\langle w_{n},\phi_f \rangle) \right]\mathrm{d}x\right.\\ 
\left.  + \int_{\mathbb{R}^{d}}\lbrace 1-\overline{w}_{n}(x)\rbrace
\left[ F(\langle \theta^{-}_{x,R_{n},u_{n}}(w_{n}),\phi_f \rangle) 
- F(\langle w_{n},\phi_f \rangle) \right]\mathrm{d}x\right),
\end{multline}
and, writing ${\vs}_n=n^\alpha \tilde{\vs}_n=n^\alpha{\vs}/n^{2/3}$, 
\begin{align*}
\mathcal{L}^{fsel}_{n}\psi_{F,f}(w_n,\zeta) &=
n^{1+\frac{d}{3}}\tilde{\vs}_n\int_{\mathbb{R}^{d}}
\left\lbrace\overline{w}_{n}(x) -\zeta(x)
\overline{w}_{n}(x)(1-\overline{w}_{n}(x))\right\rbrace \\
&\phantom{\mathcal{L}^{fsel}_{n}F(\langle w_{n},\phi_f \rangle)}\times \left[ F(\langle \theta^{+}_{x,R_{n},u_{n}}(w_{n}),\phi_f \rangle) - 
F(\langle w_{n},\phi_f \rangle) \right]\mathrm{d}x \nonumber\\
&+n^{1+\frac{d}{3}}\tilde{\vs}_{n}
\int_{\mathbb{R}^{d}}\left\lbrace(1-\overline{w}_{n}(x)) + 
\zeta(x)\overline{w}_{n}(x)(1-\overline{w}_{n}(x)) \right\rbrace\\
&\phantom{\mathcal{L}^{fsel}_{n}F(\langle w_{n},\phi_f \rangle)}\times \left[ F(\langle \theta^{-}_{x,R_{n},u_{n}}(w_{n}),\phi_f \rangle) 
- F(\langle w_{n},\phi_f \rangle) \right]\mathrm{d}x.
\end{align*}

Just as in the non-spatial case, in order to reduce the generator 
to the form~\eqref{kurtz rearrangement},
we appeal to the approach of \cite{kurtz:1973}.
Entirely analogously to our calculations of Section~\ref{nospace} (see \eqref{test_function_g}), we 
evaluate the generator on test functions of the form
\begin{equation}
\label{test function for Kurtz trick}
G_{F,f}(w_{n},\zeta) = \psi_{F,f}(w_{n}) 
+ n^{-\alpha}\mathcal{L}^{fsel}_{n}\psi_{F,f}(w_n,\zeta),
\end{equation}
where $F\in C^\infty(\IR)$ and $f\in C_c^\infty(\IR^d)$.
The second term on the right hand side will form part of $\epsilon_n^f$
in~(\ref{kurtz rearrangement}).
Evidently we must now account for the changing environment, but this is
just the generator of a pure jump process in which a new state is sampled
from a (mean zero) stationary distribution at the times of a Poisson process of 
rate $n^{2\alpha}\tau^{env}$.

Writing the full generator as ${\cal L}^{neu}+n^\alpha{\cal L}^{fsel}+
n^{2\alpha}{\cal L}^{env}$, as in~\eqref{LFVres},
we observe that ${\cal L}^{env}\psi_{F,f}=0$ (since $\psi_{F,f}$ is
independent of $\zeta$). 
In~\eqref{Evaluation of Lfsel on F} below, we show that 
$${\cal L}^{fsel}\psi_{F,f}(w_n,\zeta)
= -{\vs}\bar{u}V_R F^{\prime}(\langle w_n,\phi_f \rangle)
\int \zeta(x)\overline{w}_n(x)
\left(1-\overline{w}_n(x)\right)f(x)\mathrm{d}x
+\mathcal{O}\left(n^{-\frac{1}{3}}\right),$$
and since $\IE_\pi[\zeta]=0$,
$\IE_\pi[{\cal L}^{fsel}\psi_{F,f}] =\mathcal{O}(n^{-1/3})$.
Calculating exactly as in the non-spatial case (see \eqref{PH trick, non-spatial}), we find a spatial analogue of \eqref{limgen}, that is
\begin{align}\label{Limit formal evaluation step 1}
 \mathcal{L}_{n}G_{F,f}(w_{n},\zeta)+  
\mathcal{O}\left( n^{-\alpha}\right) &=  
\mathcal{L}^{neu}_{n}\psi_{F,f}(w_{n}) + 
\mathcal{L}^{fsel}_{n}
\left(\mathcal{L}^{fsel}_{n}\psi_{F,f}\right)(w_n,\zeta) 
\nonumber\\
& := \mathcal{L}^{neu}_{n}\psi_{F,f}(w_{n}) + 
\mathcal{G}_{n}\psi_{F,f}(w_{n},\zeta),
\end{align}
where we 
have used that $\alpha<1/6<1/3$ to absorb $\IE_\pi[{\cal L}^{fsel}\psi_{F,f}]$ 
into the error on the left hand side.
In this notation, for $F \in C^{\infty}(\mathbb{R})$ and 
$f \in C^{\infty}_{c}(\mathbb{R}^{d})$
\begin{align*}
\psi_{F,f}(w_n(t)) - 
\int_{0}^{t}\left[\mathcal{L}_{n}^{neu}
\psi_{F,f}\left(w_{n}(s)\right) 
+ \mathcal{G}_{n}\psi_{F,f}\left(w_{n}(s), \zeta(s)\right)\right]
\mathrm{d}s + \mathcal{O}\left( n^{-\alpha}\right)
\end{align*}
is a martingale. 

\subsubsection*{Neutral part}

First we find an approximation to the part of the generator
corresponding to neutral events. This mirrors the proof of Theorem~1.8 of
 \cite{etheridge/veber/yu:2014} and so we recall the strategy, but omit
many of the details. 

Taking the expression~(\ref{neutral bit}) for ${\cal L}_n^{neu}$,
we perform a Taylor expansion of $F$ about $\langle w_n, \phi_{f} \rangle$ 
to obtain
\begin{align*}
\int_{0}^{t}\mathcal{L}_{n}^{neu}\psi_{F,f}\left(w_{n}(s)\right)\mathrm{d}s 
=
(1 - \vs_{n})
  \int_{0}^{t}\left[ A_{n}(s) + B_{n}(s) + C_{n}(s)\right]\mathrm{d}s
   ,
\end{align*}
where $A_{n}(s)$, $B_{n}(s)$, $C_{n}(s)$ can be expressed as
\begin{align*}
A_{n}(s) &= u_{n}n^{1 + \frac d3}F'\left(\langle\overline{w}_{n}(s),f \rangle\right)
\int_{\mathbb{R}^{d}}\left[\overline{w}_{n}(s,x) 
\left\langle \ind_{B_{n}(x)}\left(1-w_{n}(s)\right),
\phi_{f}\right\rangle \right. \\
&\phantom{{}= u_{n}n^{1 + \frac d3}F'\left(\langle\overline{w}_{n}(t),f \rangle\right)\int_{\mathbb{R}^{d}}}
\left.-\left(1 - \overline{w}_{n}(s,x)\right)\left\langle \ind_{B_{n}(x)}w_{n}(s),
\phi_{f}\right\rangle \right] \mathrm{d}x\\
B_{n}(s) &= u_{n}^{2}n^{1 + \frac d3}\frac{F''\left(\langle\overline{w}_{n}(s),f \rangle\right)}{2}\int_{\mathbb{R}^{d}}\left[\overline{w}_{n}(s,x) 
\left\langle \ind_{B_{n}(x)}\left(1-w_{n}(s)\right),\phi_{f}\right\rangle^{2} 
\right. \\
&\phantom{{}= u_{n}n^{1 + \frac d3}F'\left(\langle\overline{w}^{n}(t),f \rangle\right)\int_{\mathbb{R}^{d}}}
\left.+\left(1 - \overline{w}_{n}(s,x)\right)\left\langle \ind_{B_{n}(x)}w_{n}(s),
\phi_{f}\right\rangle^{2} \right] \mathrm{d}x\\
C_{n}(s) &\leq \mathcal{O}(n^{-2d/3}),
\end{align*}
where we have used \eqref{theta+ w} and \eqref{theta- w} and $C_n$
corresponds to the third order term in the Taylor expansion.

The key to controlling $A_n$ and $B_n$ is a Taylor expansion of $\phi_f$.
First we use Fubini's Theorem to write
\begin{align*}
A_{n}(s) &= 
u_{n}n^{1 + \frac d3}F'\left(\langle\overline{w}_{n}(s),f \rangle\right)\int_{\mathbb{R}^{d}}\left\lbrace \overline{w}_{n}(s,x) \langle\ind_{B_{n}(x)},\phi_{f}\rangle  -  \langle \ind_{B_{n}(x)}w_{n}(s),\phi_{f}\rangle\right\rbrace\mathrm{d}x
\\
& = u_{n}n^{1 + \frac d3}F'\left(\langle\overline{w}_{n}(s),f \rangle\right)
\int_{\mathbb{R}^{d}}\left[\frac{n^{\frac d3}}{V_{R}} 
\int_{\IR^d}\int_{\IR^d} \ind_{\{|x-y|<R_{n}\}}w_n(s,y)\ind_{\{|x-z|<R_{n}\}}\phi_{f}(z)\mathrm{d}y\mathrm{d}z \right.\\ 
&\phantom{= u_{n}n^{1 + \frac d3} u_{n}n^{1 + \frac d3} u_{n}n^{1 + \frac d3}F'\left(\langle\overline{w}_{n}(s),f \rangle\right)
\int_{\mathbb{R}^{d}}}  \left.- \int_{\IR^d} \ind_{\{|x-y|<R_{n}\}}w_n(s,y)\phi_{f}(y)\mathrm{d}y\right] \mathrm{d}x \\
&= u_{n}n^{1 + \frac d3}F'\left(\langle\overline{w}_{n}(s),f \rangle\right)
\int_{\mathbb{R}^{d}}\frac{n^{\frac d3}}{V_{R}} \int_{\IR^d}\int_{\IR^d} 
\ind_{\{|x-y|<R_{n}\}}w_n(s,y)\ind_{\{|x-z|<R_{n}\}}\\
&\phantom{u_{n}n^{1 + \frac d3}F'\left(\langle\overline{w}_{n}(s),f \rangle\right)
\int_{\mathbb{R}^{d}}\frac{n^{\frac d3}}{V_{R}} \int_{\IR^d}\int_{\IR^d} 
\ind_{\{|x-y|<R_{n}\}}w_n(s,y)}\left(\phi_{f}(z)-\phi_{f}(y)\right)\mathrm{d}y\mathrm{d}z\mathrm{d}x.
\end{align*}
Now Taylor-expand $\phi_{f}$ around $y$ to obtain
\begin{multline}\label{hesjany}
A_{n}(s) = u_{n}n^{1 + \frac d3}F'\left(\langle\overline{w}_{n}(s),f \rangle\right)
\int_{\mathbb{R}^{d}}\left[\frac{n^{\frac d3}}{V_{R}} 
\int\int \ind_{\{|x-y|<R_{n}\}}w_n(s,y)\ind_{\{|x-z|<R_{n}\}}\right. 
\\
 \left.\times \left[D\phi_{f}(y)(z-y) + \frac12 (z-y){\mathbf H}(y)\phi_{f}(z-y) + \mathcal{O}\left(|z-y|^{3}\right)\ind_{y\in S_{f}}  \right]\right]\mathrm{d}y\mathrm{d}z\mathrm{d}x,
\end{multline}
where ${\mathbf H}$ denotes the Hessian. 
Using anti-symmetry of $z-y$, the term involving the gradient vanishes, as
do the off-diagonal terms in the Hessian, so 
that~\eqref{hesjany} is equal to
\begin{align*}
u_{n}n^{1 + \frac d3}F'\left(\langle\overline{w}_{n}(s),f \rangle\right)
\int_{\mathbb{R}^{d}}\frac{n^{\frac d3}}{V_{R}} 
\int_{\mathbb{R}^{d}}\int_{\mathbb{R}^{d}} 
\ind_{\{|x-y|<R_{n}\}}\ind_{\{|x-z|<R_{n}\}}w_{n}(s,y)\\
\times\frac12 \sum_{i=1}^{d}(z_{i} - y_{i})^{2}\frac{\partial^{2}}{\partial y_{i}^{2}}\phi_{f}(y) \mathrm{d}y\mathrm{d}z\mathrm{d}x,
\end{align*}
plus a lower order term which is bounded uniformly in the time variable $s$ by 
$\mathtt{Vol}(S_f)u_nC_f$ (where $C_f$ is a bound on the third
derivative of $f$ and $S_f$ is the support of $f$).
The analogue of~\eqref{taylor of phi} with $\phi_{f}$ replaced by  
$\frac{\partial^2 \phi_f}{\partial y_i^2}$ is
$$ \frac{\partial^2 }{\partial y_i^2}\phi_f(y)
=\frac{\partial^2 f}{\partial y_i^2}(y)+\mathcal{O}(n^{-2/3}
\ind_{B_n(y)\cap S_f\neq\emptyset}),$$
and (using~(\ref{<w,phif> = <barw,f>})
with $\Delta f$ in place of $f$) we deduce that
\begin{align}
A_{n}(s) & = u_{n}n^{1 + \frac d3}F'
\left(\langle\overline{w}_{n}(s),f \rangle\right)\int_{\mathbb{R}^{d}}
\frac{n^{\frac d3}}{V_{R}} \int_{\mathbb{R}^{d}}\int_{\mathbb{R}^{d}} 
\ind_{\{|x-y|<R_{n}\}}\ind_{\{|x-z|<R_{n}\}}w_{n}(s,y) 
\nonumber \\
&\phantom{A_{n}(s)  = u_{n}n^{1 + \frac d3}} \times\frac12 \sum_{i=1}^{d}(z_{i} - y_{i})^{2}
\frac{\partial^{2}}{\partial y_{i}^2}{f}(y) 
\mathrm{d}y\mathrm{d}z\mathrm{d}x + \mathcal{O}\left(n^{-\frac23}C\right)\nonumber\\
&=  \frac{\bar{u}\Gamma_R}{2}F'\left(\langle\overline{w}_{n}(s),f \rangle\right)
\int_{\mathbb{R}^{d}} w_{n}(s,y)\Delta f\mathrm{d}y  + 
\mathcal{O}\left(n^{-\frac23}C\right) \nonumber \\ 
&= \frac{\bar{u}\Gamma_R}{2}F'\left(\langle\overline{w}_{n}(s),f \rangle\right)
\langle \overline{w}_{n}(s),\Delta f \rangle  + 
\mathcal{O}\left(n^{-\frac23}C\right)
\end{align}
where
\begin{equation}
\label{defn of Gamma}
\Gamma_R= \frac{1}{V_R}\int_{B(0,R)}\int_{B(x,R)}(z_1)^2\mathrm{d}z\mathrm{d}x,
\end{equation}
and $z_1$ is the first component of the vector $z\in\IR^d$.
The error term tends to zero uniformly in $s$ and since 
$|\langle \overline{w}_n(s),f\rangle|\leq \|f\|_\infty \mathtt{Vol}(S_f)$,
we have that $|A_n(s)|$ is uniformly bounded in $s$ and $n$.
The contribution coming from $B_{n}(s)$ can be treated similarly:
\begin{multline}\label{estimate for Bn(s)}
B_{n}(s) =
 u_{n}^{2}n^{1 + \frac d3}
\frac{F''\left(\langle\overline{w}_{n}(s),f \rangle\right)}{2} 
\\
\times
\int_{\mathbb{R}^{d}}\overline{w}_{n}(s,x)\left\langle\ind_{B_{n}(x)}\big(1 - w_{n}(s)\big),
\phi_{f}\right\rangle^{2} + (1 - \overline{w}_{n}(s,x))\left\langle\ind_{B_{n}(x)}{w}_{n}(s),\phi_{f}\right\rangle^{2} \mathrm{d}x
\\
=
u_{n}^{2}n^{1 + \frac d3}
\frac{F''\left(\langle\overline{w}_{n}(s),f \rangle\right)}{2} 
\int_{\mathbb{R}^{d}}
\big\lbrace
\overline{w}_{n}(s,x)\left\langle\ind_{B_{n}(x)},
\phi_{f}\right\rangle^{2}
\\
-
2\overline{w}_{n}(s,x)\left\langle\ind_{B_{n}(x)},
\phi_{f}\right\rangle
\left\langle w_{n}(s)\ind_{B_{n}(x)},
\phi_{f}\right\rangle
+
\left\langle w_{n}(s)\ind_{B_{n}(x)},
\phi_{f}\right\rangle^{2}
\bigg\rbrace
\mathrm{d}x
\\
=
u_{n}^{2}n^{1 + \frac d3}
\frac{F''\left(\langle\overline{w}_{n}(s),f \rangle\right)}{2} 
\int_{\mathbb{R}^{d}}
\bigg\lbrace
\overline{w}_{n}(s,x)\left\langle\ind_{B_{n}(x)},
\phi_{f}\right\rangle
\big(
\left\langle\ind_{B_{n}(x)},
\phi_{f}\right\rangle
-
\left\langle\ind_{B_{n}(x)}w_{n}(s),
\phi_{f}\right\rangle
\big)
\\
-
\left\langle\ind_{B_{n}(x)}w_{n}(s),
\phi_{f}\right\rangle
\big(\overline{w}_n(s,x)
\left\langle\ind_{B_{n}(x)},
\phi_{f}\right\rangle
-
\left\langle\ind_{B_{n}(x)}w_{n}(s),
\phi_{f}\right\rangle
\big)
\bigg\rbrace
\mathrm{d}x
\\
\leq u_{n}^{2}n^{1 + \frac d3}\frac{F''\left(\langle\overline{w}_{n}(s),
f \rangle\right)}{2} \int_{S_f}2\mathtt{Vol}\left( B_{n}(x)\right)^{2}\|f\|_{\infty}^{2}\mathrm{d}x \leq Cn^{\frac{1-d}{3}},
\end{multline}
which implies that  the sequence of quadratic variations is not 
only tight, but also, for dimensions $d \geq 2$, it tends to $0$ 
uniformly on compact time sets. 
In other words, any randomness remaining in the limit in $d\geq 2$ will be
due to the fluctuations in the environment. 

In $d = 1$, using~\eqref{taylor of phi}
and a further Taylor expansion to approximate $\phi_f(y)$ by $f(x)$ for
$y\in B_n(x)$ (up to an error of order $n^{-1/3}$) and that 
$n^{1/3}u_n=\overline{u}$, $n^{1/3}\mathtt{Vol}(B_n)=V_R$, we have
\begin{equation}
\label{Bn(s) in d=1}
B_n(s)=\overline{u}^2V_R^2
\frac{F''\left(\langle\overline{w}_{n}(s), f \rangle\right)}{2} 
\int_{\mathbb R}\overline{w}_n(s,x)\left(1-\overline{w}_n(s,x)\right)f(x)^2\mathrm{d}x
+
\mathcal{O}\left(n^{-\frac13}\right).
\end{equation}
It is this term that (via an application of 
Proposition~\ref{continuity proposition}) leads to the final term 
in~\eqref{Martingale Problem for original model 1d}.

\subsubsection*{Fluctuating part}

In order to tackle the part of the generator which describes the 
fluctuating selection, we need to understand the action of the operator 
$\mathcal{L}^{fsel}_{n}$ on 
$\mathcal{L}^{fsel}_{n}\psi_{F,f}(w_n,\zeta)$.
It is here that we must diverge from \cite{etheridge/veber/yu:2014}.

As before, 
although the test function $\psi_{F,f}(w_n)$ does not depend 
on the environment, the result of applying 
$\mathcal{L}^{fsel}_{n}$ to it is 
a function that {\em does} depend on the environment, to which
we must then apply $\mathcal{L}^{fsel}_n$. Recall that 
$\tilde{\vs}_n=\vs/n^{2/3}$.
We take a Taylor expansion of $F$ about $\langle w_n,\phi_f\rangle$ to obtain 
\begin{align}\label{Lfsel n evaluation}
&\mathcal{L}^{fsel}_{n}\psi_{F,f}(w_{n},\zeta) + 
\mathcal{O}\left( n^{-\frac{1}{3}} C_f\right) \nonumber\\
&= n^{1 + \frac{d}{3}}\tilde{\vs}_{n}u_{n}\left[\int_{\mathbb{R}^{d}}
\left\lbrace\overline{w}_{n}(x)  -\zeta(x)\overline{w}_{n}(x)(1-\overline{w}_{n}(x))
\right\rbrace F^{\prime}(\langle w_n,\phi_{f} \rangle)
\langle \ind_{B_{n}(x)}(1-w_n),\phi_{f} \rangle\mathrm{d}x  \right.\nonumber\\
& -\left.\int_{\mathbb{R}^{d}}\left\lbrace(1-\overline{w}_{n}(x)) + 
\zeta(x)\overline{w}_{n}(x)(1-\overline{w}_{n}(x))\right\rbrace 
F^{\prime}(\langle w_n,\phi_f \rangle)\langle \ind_{B_{n}(x)}w_{n},\phi_{f} \rangle\mathrm{d}x  \right], 
\end{align}
(plus lower order terms)
where we have used \eqref{theta+ w} and \eqref{theta- w}.  

Now using our scaling, \eqref{taylor of phi} and 
\eqref{<w,phif> = <barw,f>}, we obtain
\begin{align}
&\mathcal{L}^{fsel}_{n}\psi_{F,f}(\overline{w}_{n},\zeta) \nonumber\\
&= n^{\frac{d}{3}}{\vs}\bar{u}F^{\prime}(\langle w_n,\phi_f \rangle)
\left[\int_{\mathbb{R}^{d}}\lbrace\overline{w}_{n}(x) - \zeta(x)
\overline{w}_{n}(x)(1-\overline{w}_{n}(x))\rbrace\right. \nonumber \\
&\phantom{n^{\frac{d}{3}}suF^{\prime}(\langle w,\phi \rangle,\zeta) n^{\frac{d}{3}}suF^{\prime}(\langle w,\phi \rangle)}
\times \left\langle \ind_{B_{n}(x)}(1-w_n),f + 
\mathcal{O}\left(|x-y|\right)\ind_{x\in S_f}\right\rangle\mathrm{d}x  \nonumber\\
&-\left.\int_{\mathbb{R}^{d}}\left\lbrace(1-\overline{w}_{n}(x)) + 
\zeta(x)\overline{w}_{n}(x)(1-\overline{w}_{n}(x))\right\rbrace\left\langle 
\ind_{B_{n}(x)}w_n,\left[f + \mathcal{O}\left(|x-y|\right)\ind_{x\in S_f}\right] \right\rangle\mathrm{d}x  
\right] \nonumber \\
&\phantom{AAAAAAAAAAAAAAAAAAAAAAAAAAAAAAAAAAAAAA}
+ \mathcal{O}\left(n^{-\frac{1}{3}}\right)\nonumber\\  
&= -{\vs}\bar{u}V_R F^{\prime}(\langle w_{n},\phi_f \rangle)
\int_{\mathbb{R}^{d}}\zeta(x)\overline{w}_{n}(x)(1-\overline{w}_{n}(x))f(x)
\mathrm{d}x + \mathcal{O}\left(n^{-\frac{1}{3}}\right) \nonumber \\
&= -{\vs}\bar{u}V_R F^{\prime}(\langle w_n,\phi_f \rangle)
H_f(\overline{w}_{n},\zeta) + 
\mathcal{O}\left(n^{-\frac{1}{3}}\right),
\label{Evaluation of Lfsel on F}
\end{align}
where 
$$H_f(\overline{w}_n, \zeta)=\int \zeta(x)\overline{w}_n(x)
\left(1-\overline{w}_n(x)\right)f(x)\mathrm{d}x.$$
We now need to evaluate ${\cal L}_n^{fsel}$ on the product
in~(\ref{Evaluation of Lfsel on F}).
\begin{multline}
\label{sum of two terms}
\mathcal{L}^{fsel}_{n}\left(\mathcal{L}^{fsel}_{n}\psi_{F,f}\right)
(w_n,\zeta) + \mathcal{O}\left( n^{-\frac{1}{3}}\right)  
 \\
 = 
-{\vs}\bar{u}V_R\mathcal{L}^{fsel}_{n}\left(\psi_{F^{\prime},f}
\right)(w_n,\zeta) 
H_f(\overline{w}_{n},\zeta) - 
{\vs}\bar{u}V_R F^{\prime}(\langle w_n,\phi_f \rangle)
\left(\mathcal{L}^{fsel}_{n}
H_f\right)(\overline{w}_{n},\zeta). 
\end{multline}
As in the non-spatial case, these two terms will correspond to the 
diffusion and drift terms respectively in the stochastic p.d.e..

The difficulty that we now face is that $H$ depends on 
$\langle\overline{w}^2,f\rangle$ as well as 
$\langle\overline{w},f\rangle$ and it will no longer suffice to use the 
test function $\phi_f$ and integrate by parts. 
Let us consider the effect on $\overline{w}_n(x)^2$ of an event in the
ball $B_n(y)$ in which the parent is of type $a$. 
After the event $\overline{w}_n(x)^2$ is replaced by
\begin{multline}
\label{change in wsquared}
\left(\overline{w}_n(x)+\frac{n^{d/3}}{V_R}\int_{B_n(y)} 
\ind_{\{|z-x|<R_n\}}u_n(1-w_n(z))\mathrm{d}z\right)^2
\\
=\overline{w}_n(x)^2
+2u_n\overline{w}_n(x)\frac{n^{d/3}}{V_R}\int_{B_n(y)}\vint_{B_n(y)} 
\ind_{\{|z-x|<R_n\}} (1-w_n(z))\mathrm{d}z
\\
+\mathcal{O}(u_n^2\ind_{\{|x-y|<2R_n\}}).
\end{multline}

The corresponding change in $\overline{w}_{n}(x)(1-\overline{w}_{n}(x))$ is then
$$(1-2\overline{w}_n(x))\frac{n^{d/3}}{V_R}
\int_{B_n(y)}\ind_{\{|z-x|<R_n\}}
u_n(1-w_n(z))\mathrm{d}z
+\mathcal{O}(u_n^2\ind_{\{|x-y|<2R_n\}}).
$$
If the parent is type $A$, then the change is
$$-(1-2\overline{w}_n(x))
\frac{n^{d/3}}{V_R}\int_{B_n(y)}\ind_{\{|z-x|<R_n\}} 
u_nw_n(z)\mathrm{d}z
+\mathcal{O}(u_n^2)\ind_{\{|x-y|<2R_n\}}).$$
To evaluate the second term in~\eqref{sum of two terms},
we calculate
the effect of an event in $B_n(y)$ and integrate with respect to $y$,
taking into account only the first order change in $\overline{w}_n^2(x)$
as given in~\eqref{change in wsquared}. 
We can now observe that the second term in \eqref{sum of two terms} can be written as 
\begin{align}
{\cal L}&_n^{fsel}\Big(\int_{\IR^d} \zeta(x)\overline{w}_n(x)\left(1-\overline{w}_n(x)\right)
f(x)\mathrm{d}x\Big) \nonumber
\\
&=n^{1+\frac{d}{3}}\tilde{\vs}_n
\int_{\IR^d}\int_{\IR^d}\int_{\IR^d}
\lbrace \overline{w}_n(y)-\zeta(y)\overline{w}_n(y)
\left(1-\overline{w}_n(y)\right)\rbrace
\zeta(x)
\nonumber
\\ 
&\phantom{n^{1+\frac{d}{3}}n^{1+\frac{d}{3}}}\times 
\frac{n^{d/3}u_n}{V_R}
\ind_{\{|x-z|<R_n\}}
\ind_{\{|y-z|<R_n\}}\left(1-w_n(z)\right)\left(1-2\overline{w}_n(x)\right)
f(x)\mathrm{d}z\mathrm{d}y\mathrm{d}x
\nonumber
\\
&-n^{1+\frac{d}{3}}\tilde{\vs}_n
\int_{\IR^d}\int_{\IR^d}\int_{\IR^d}\lbrace (1-\overline{w}_n(y))+\zeta(y)\overline{w}_n(y)
\left(1-\overline{w}_n(y)\right)\rbrace
\zeta(x)
\nonumber
\\
&\phantom{n^{1+\frac{d}{3}}}\times 
\frac{n^{d/3}u_n}{V_R}
\ind_{\{|x-z|<R_n\}}
\ind_{\{|y-z|<R_n\}}w_n(z)\left(1-2\overline{w}_n(x)\right)
f(x)\mathrm{d}z\mathrm{d}y\mathrm{d}x +{\mathcal O}(n^{-\frac13})
\label{eq 8.23.5}
\end{align}
We can further simplify this expression to
\begin{align}
&{\cal L}_n^{fsel}\Big(\int_{\IR^d} \zeta(x)\overline{w}_n(x)\left(1-\overline{w}_n(x)\right)
f(x)\mathrm{d}x\Big) \nonumber
\\
&=-n^{1+\frac{d}{3}}\tilde{\vs}_n\frac{ n^{d/3}u_n}{V_R}
\int_{\IR^d}\int_{\IR^d}\int_{\IR^d} \zeta(x)\zeta(y)
 \ind_{\{|x-z|<R_n\}}\ind_{\{|y-z|<R_n\}} 
 \nonumber
\\
&\phantom{n^{1+\frac{d}{3}}n^{1+\frac{d}{3}}blablablabla}
\times \left(1-2\overline{w}_n(x)\right)f(x)
\overline{w}_n(y)\left(1-\overline{w}_n(y)\right)
\mathrm{d}z\mathrm{d}y\mathrm{d}x 
\nonumber
\\
&+
n^{1+\frac{d}{3}}\tilde{\vs}_n\frac{ n^{d/3}u_n}{V_R}
\int_{\IR^d}\int_{\IR^d}\int_{\IR^d} \bigg\lbrace\zeta(x)\ind_{\{|x-z|<R_n\}}\ind_{\{|y-z|<R_n\}}
\nonumber
\\
& \phantom{blablablablablabalbla}\times\left(1-2\overline{w}_n(x)\right)f(x)\left(\overline{w}_n(y)-
w_n(z)\right)\bigg\rbrace\mathrm{d}z\mathrm{d}y\mathrm{d}x+
{\mathcal O}(n^{-\frac13})
\nonumber
\\
&= -
\bar{u}\vs V_R\int \zeta(y)^2
\overline{w}_n(y)\left(1-\overline{w}_n(y)\right)
\left(1-2\overline{w}_n(y)\right)f(y)\mathrm{d}y
\nonumber
\\
&+
\bar{u}\vs V_R\int \vint_{B_n(y)}\zeta(y)^2
2\left(\overline{w}_n(x)-(\overline{w}_n(y)\right)
\overline{w}_n(y)
\left(1-\overline{w}_n(y)\right)
f(y)\mathrm{d}y
\nonumber
\\
&+\bar{u}\vs V_R\int \vint_{B_n(y)}\Big(\zeta(y)^2-\zeta(x)\zeta(y)\Big)
\overline{w}_n(y)\left(1-\overline{w}_n(y)\right)
\left(1-2\overline{w}_n(x)\right)f(y)\mathrm{d}x\mathrm{d}y
\label{tricky term in lsel}\\
&\phantom{+\bar{u}\vs V_R\int \vint_{B_n(x)}\Big(\zeta(x)^2-\zeta(x)\zeta(y)\Big)
\overline{w}_n(x)\left(1-2\overline{w}_n(x)\right)} 
+
{\mathcal O}(n^{-\frac13})
,
\nonumber
\end{align}
where we have written 
$$(1-2\overline{w}_n(x))=(1-2\overline{w}_n(y))
+2(\overline{w}_n(y)-\overline{w}_n(x)),$$
performed a Taylor expansion of $f$ around $y$ for $|y-x|<R_n$, and
used Fubini's Theorem.

Since $F'(\langle w_n,\phi_f\rangle)$ is independent of $\zeta$, the first term
in~\eqref{sum of two terms} can be read off immediately from our 
previous calculations, and combining all the above, we see that we should 
like to take $A$ in Theorem~\ref{kurtzaveraging} to be 
\begin{multline}
\label{form_of_A}
AF(\langle \overline{w},f\rangle)=
F'(\langle \overline{w},f\rangle)\Big\{\langle\overline{w},
\frac{\bar{u}\Gamma_R\Delta f}{2}\rangle+
\langle \bar{u}^2\vs^2V_R^2\overline{w}(1-\overline{w})(1-2\overline{w}),f
\rangle\Big\}\\
+
F''(\langle \overline{w},f\rangle)\Big\{
\bar{u}^2\vs^2V_R^2\langle \zeta\overline{w}(1-\overline{w},f\rangle^2
+\ind_{d=1}\frac{\bar{u}^2V_R^2}{2}\langle \overline{w}(1-\overline{w}), f^{2}\rangle\Big\},
\end{multline}
and
\begin{multline}
\label{epsilonnf}
\epsilon_n^f(t)= 
\bar{u}^2\vs^2 V_R^2 \int_{0}^{t}\bigg\lbrace F'(\langle\overline{w}_n(s),f\rangle)
\int_{\mathbb{R}^{d}} \vint_{B_n(y)}\Big(\zeta(s,y)^2-\zeta(s,x)\zeta(s,y)\Big)
\\
\times \overline{w}_n(s,y)\left(1-\overline{w}_n(s,y)\right)
\left(1-2\overline{w}_n(s,x)\right)f(y)\bigg\rbrace\mathrm{d}x\mathrm{d}y\mathrm{d}s
\\
+\bar{u}^2\vs^2 V_R^2 \int_{0}^{t}\bigg\lbrace F'(\langle\overline{w}_n(s),f\rangle)
\int_{\mathbb{R}^{d}} \vint_{B_n(y)}\Big(\overline{w}_n(s,x)-
\overline{w}_n(s,y)\Big)
\\
\times \overline{w}_n(s,y)\left(1-\overline{w}_n(s,y)\right)
f(y)\bigg\rbrace\mathrm{d}x\mathrm{d}y\mathrm{d}s
+Ct{\mathcal O}(
n^{-\alpha}),
\end{multline}
where the term of order $n^{-\alpha}$ comes from the second term
on the right of~(\ref{test function for Kurtz trick}) and we have
used that $\alpha<1/6$ to absorb a term of order $n^{-1/3}$.
The additional term in $A$ in $d=1$ stems from the term $B_n(s)$ 
in~(\ref{Bn(s) in d=1}).
Since $A$ is uniformly bounded,~(\ref{kurtzbound})
is obviously satisfied.
To check that $\IE[\sup_{t\leq T}|\epsilon_n^f(t)|]\rightarrow 0$ as 
$n\rightarrow\infty$, we must control the integral terms in~(\ref{epsilonnf}).
We need to use the continuity of $\overline{w}_n(x)$. As shown in 
the proof of Proposition~\ref{continuity proposition} 
(see Appendix~\ref{continuity estimates}, \eqref{ctty bound eq}), 
for any fixed constant $\lambda > 0$,
\begin{equation}
\label{ctty eq}
\IE\left[\sup_{s\leq t}|\overline{w}_n(s,x)-\overline{w}_n(s,y)|\right]\leq Cn^{-1/12}e^{\lambda |x|},
\qquad\mbox{ for }|x-y|<2R_n.
\end{equation}
This immediately controls the expectation of the supremum of 
the second integral in~\eqref{epsilonnf}.

To control the first integral in~\eqref{epsilonnf}, first observe that
for $|x-y|<R_n$, by~(\ref{regularity of g}), 
$$\IE[\zeta(y)^2-\zeta(x) \zeta(y)]=\Big(g(y,y)-g(x,y)\Big)\leq CR_n,$$
so that writing
$$I(s)=
\int \vint_{B_n(y)}\Big(\zeta(s,y)^2-\zeta(s,x)\zeta(s,y)\Big)
\overline{w}_n(s,y)\left(1-\overline{w}_n(s,y)\right)
\left(1-2\overline{w}_n(s,y)\right)f(y)\mathrm{d}y,$$
we have
$$\IE_\pi\left[I(s)\right]
\leq C n^{-1/3}.$$
Let us write $\tau_i, i=1,\ldots ,N_T$ for the times at which the environment
is resampled in the time interval $[0,T]$, and setting $\tau_0=0$ and
$\tau_{N_T+1}=T$, write 
$$S_i=\sup\{I(s): s\in (\tau_{i-1}, \tau_{i}]\},\quad i=1,\ldots, N_T+1.$$
By Markov's inequality, $\IP[S_i>1/n^\delta]\leq Cn^{-1/3+\delta}$ and 
$$\IP\left[\sup_{1\leq i\leq N_T+1}S_i>\frac{1}{n^\delta}\right]\leq 
\IE\left[N_T+1\right]
\IP\left[S_1>\frac{1}{n^\delta}\right]\leq Cn^{2\alpha-\frac{1}{3}+\delta}T.$$
Thus for $\delta>0$, 
$$\IE\left[\sup_{1\leq i\leq N_{T}+1}S_i\right]\leq C\left(\frac{1}{n^\delta}
+
\IP\left[\sup_{1\leq i\leq N_T+1}S_i>\frac{1}{n^\delta}\right]\right)
\leq C\left(\frac{1}{n^\delta} + n^{2\alpha+\delta -\frac{1}{3}}\right).$$
Since $\alpha<1/6$, we may choose $\delta>0$ in such a way that the right
hand side tends to zero, and so
we conclude that for fixed $T$,
$$\IE\left[\sup_{t\leq T}|\epsilon_n^f(t)|\right]\rightarrow 0$$
as $n\rightarrow\infty$, as required.

\begin{remark}
\label{breakdown for white noise}
The estimate on $\IE\left[\sup_{t\leq T}|\epsilon_n^f(t)|\right]$ 
depends on the Lipschitz bound on the correlation kernel $g(x,y)$. 
This is precisely the place where our argument breaks down if we allow 
the environmental noise to be `white'.
\end{remark}

The only difficulty that we still face is that the compact
containment condition, and hence tightness of our sequence of 
processes, is in the space ${\cal M}_\lambda$, and since 
$\langle \overline{w}_n^2,f\rangle$ cannot be written as an integral with
respect to the measure $\overline{M}_n$ or $\overline{M}_n^{\otimes 2}$, convergence of 
$\langle \overline{w}_n^2,f\rangle$ to $\langle (w^\infty)^2,f\rangle$
is not a simple consequence of the convergence of the sequence of 
measure-valued evolutions. Evidently, the term in $\overline{w}_n^3$
suffers from the same problem. 
However, the proof of Theorem~\ref{scaling result} will be complete, if we 
can prove the following proposition.
\begin{proposition}
\label{continuity proposition}
For $k=2,3$, we have, for each fixed $s$,
\begin{equation}
\limsup_{n\rightarrow\infty}\IE\left[\left|\langle \overline{w}_n^k(s),
f\rangle-\langle (w^\infty)^k(s),f\rangle\right|\right]=0.
\end{equation}
\end{proposition}
The proof, which we give in Appendix~\ref{continuity estimates}, once
again follows \cite{etheridge/veber/yu:2014}. It rests on 
the continuity 
estimate~\eqref{ctty eq} on the densities $\overline{w}_n$ 
that we used above, 
and integration with respect
to a suitable approximate identity.  

\subsection{Uniqueness of solutions for the limiting equation}
\label{Subsection Uniqueness of solutions for limiting equation}
In order to establish uniqueness of the limit points in $d\geq 2$, 
and hence
convergence of our rescaled SLFVFS, we consider the corresponding 
stochastic partial differential equation. The first task is to show
that any solution to the martingale problem~(\ref{Martingale Problem for original model nd})
is a weak solution to the stochastic p.d.e.~(\ref{Limiting equation nd}).
We establish this in Appendix~\ref{mgpspde}, using a technique of
 \cite{kurtz:2010}. 
This then reduces the question to 
that of the weak uniqueness of
the solution to the corresponding stochastic p.d.e.,~which we deduce 
from a pathwise uniqueness result of 
\cite{rippl/sturm:2013} 
and a suitable version of the Yamada-Watanabe Theorem.

In one dimension, although any solution to the martingale problem 
does indeed give a solution to the stochastic p.d.e., we do not
have an analogue of the result of  \cite{rippl/sturm:2013}. The 
only case in which we have a proof of uniqueness of the equation is
when $W$ is also space-time white noise, in which case we can 
use the duality of Section~\ref{duality}. We note, however, that the
proof of our main result would need to be modified to capture this
form of environmental noise, see Remark~\ref{breakdown for white noise}. 
We do not currently know how to do this.

\begin{theorem}[\cite{rippl/sturm:2013}, Theorem 1.2] 
\label{rippl sturm theorem}
Consider the stochastic heat equation
\begin{equation}\label{Heat equation}
du_{t} = \frac12 \Delta u_t \mathrm{d}t + 
\sigma(t,x,u_t)W(\mathrm{d}t,\mathrm{d}x) + b(t,x,u_t)\mathrm{d}t,
\end{equation} 
where $b$, $\sigma$ are continuous real-valued functions 
and the noise $W$ is white in 
time and coloured in space with quadratic variation given by
\begin{equation*}
\langle W(\phi)\rangle_{t} = t\int_{\mathbb{R}^{d}}
\int_{\mathbb{R}^{d}}g(x,y)\phi(x)\phi(y)\mathrm{d}x\mathrm{d}y,
\end{equation*} 
where the kernel $g(x,y)$ is bounded by a Riesz potential; that is
there exists a constant $c_1>0$ such that
\begin{equation} \label{RS condition on correlation kernel}
g(x,y) \leq c_{1}(1 + |x-y|^{\alpha})\quad\mbox{for some }\alpha \in 
(0,2\wedge d).
\end{equation}
Suppose further that
\begin{enumerate}
\item 
there exists a constant $c_{2}>0$ such that 
\begin{equation*}
|\sigma(t,x,u)| + |b(t,x,u)| \leq c_{2}(1 + |u|);
\end{equation*}
\item $\sigma(t,x,u)$ is $\gamma$-H\"older continuous w.r.t. $u$, i.e.
\begin{equation*}
|\sigma(t,x,u_{1}) - \sigma(t,x,u_{2})| \leq c_{3}(t)e^{A_{1}|x|}(1 + |u_{1}| + |u_{2}|)|u_{1}-u_{2}|^{\gamma},
\end{equation*}
where $\gamma \in (0,1)$, and $c_{3}, A_{1} > 0$;
\item 
there exists a constant $c_{4}$ such that 
\begin{equation*}
|b(t,x,u_{1}) - b(t,x, u_{2})| < c_{4}|u_{1} - u_{2}|.
\end{equation*} 
\end{enumerate}
Under the above assumptions, equation \eqref{Heat equation} with initial 
data in $C_{tem}$ has a mild solution. This solution is continuous in time
and is pathwise unique whenever
$\alpha < 2(2\gamma -1)$.
\end{theorem}
\begin{corollary}
Solutions to the stochastic p.d.e.~(\ref{Limiting equation nd})
are pathwise unique.  
\end{corollary}
\begin{proof}
It is easy to see that all the conditions of
Theorem~\ref{rippl sturm theorem} are satisfied.
The functions $b$, $\sigma$, are polynomials in $w^\infty$ and, 
since the solution itself take values in 
$[0,1]$, they are both Lipschitz. 
The absolute value of the correlation kernel is bounded,
and so condition~\eqref{RS condition on correlation kernel} 
is satisfied for any $\alpha$. 
\end{proof}
\begin{remark}
Theorem~\ref{rippl sturm theorem} guarantees the uniqueness result for the 
mild solution of \eqref{Heat equation}. However, 
under the assumptions of this Theorem, the mild and analytically weak 
formulations are actually equivalent; see the discussion on p.~1917 of 
\cite{mytnik/perkins/sturm:2006}.
\end{remark}
Uniqueness of the weak solution to \eqref{Limiting equation nd} is now a 
consequence of a Yamada-Watanabe Theorem. A suitable version 
is provided by  \cite{kurtz:2007} Theorem~3.14 and 
Lemma~2.4 (see also  \cite{rippl/sturm:2013}, 
Theorem~1.3 and \cite{kurtz:2014} Example~3.9.).

\begin{appendix}

\section{Proof of Proposition~\ref{continuity proposition}}
\label{continuity estimates}

In this section we shall prove Proposition~\ref{continuity proposition} 
in the case $k=2$. The case $k=3$ is entirely analogous.

Let $\rho_\epsilon$ be a continuous density in $\IR^d$
supported on $B_{\epsilon}(0)$. Then
\begin{eqnarray*}
\Bigg|\langle \overline{w}_n^2(s),
f\rangle &-&\langle (w^\infty)^2(s),f\rangle\Bigg| \\
&&\leq 
\left|\int_{\IR^d}f(x)\overline{w}^2_n(s,x)\mathrm{d}x
-\int_{\IR^d}\int_{\IR^d}f(x)\overline{w}_n(s,x)\overline{w}_n(s,y)
\rho_\epsilon(x-y)\mathrm{d}y\mathrm{d}x\right|\\
&&+
\left|\int_{\IR^d}\int_{\IR^d}f(x){w}^\infty(s,x){w}^\infty(s,y)
\rho_\epsilon(x-y)\mathrm{d}y\mathrm{d}x-
\int_{\IR^d}f(x)({w}^\infty(s,x))^2\mathrm{d}x
\right|\\
&& +\Bigg|
\int_{\IR^d}\int_{\IR^d}f(x)\overline{w}_n(s,x)\overline{w}_n(s,y)
\rho_\epsilon(x-y)\mathrm{d}y\mathrm{d}x\\
&&\phantom{AAAAA} 
-\int_{\IR^d}\int_{\IR^d}f(x){w}^\infty(s,x){w}^\infty(s,y)
\rho_\epsilon(x-y)\mathrm{d}y\mathrm{d}x\Bigg|.
\end{eqnarray*}
From convergence of the process in ${\cal M}_\lambda$, 
we can deduce that the third term on the right
converges to zero. We turn our attention to the first term, noticing that the second term can be estimated in the same way. Taking expectations and using
Fubini's Theorem,
\begin{eqnarray}\label{Appendix Fubini}
&&\IE\left[\left|\int_{\IR^d}f(x)\overline{w}^2_n(s,x)\mathrm{d}x
-\int_{\IR^d}\int_{\IR^d}f(x)\overline{w}_n(s,x)\overline{w}_n(s,y)
\rho_\epsilon(x-y)\mathrm{d}y\mathrm{d}x\right|\right]
\nonumber \\
&\leq&
\int_{\IR^d}\IE\left[\left|f(x)\overline{w}_n(s,x)\int_{\IR^d}
\big(\overline{w}_n(s,x)-\overline{w}_n(s,y)\big)\rho_\epsilon(x-y)\mathrm{d}y\right|\right]\mathrm{d}x
\nonumber \\
&\leq&
\|f\|_{\infty}\int_{S_{f}}\int_{B_{0}(\epsilon)}\IE[|\overline{w}_n(s,x)
-\overline{w}_n(s,y)|]\rho_\epsilon(x-y)\mathrm{d}y\mathrm{d}x,
\end{eqnarray}
where $S_{f}$ denotes support of function $f$. 

We now bound the expectation on the right hand side, by proving 
a continuity estimate for $\overline{w}_n$. In the limit as 
$n\rightarrow\infty$, this results from the smoothing effect of the heat 
semigroup.
We use a well-known
trick of  \cite{mueller/tribe:1995} based on substituting 
a careful choice of test functions,
which approximate the Brownian transition density,
into the martingale characterization of the process.  
We shall focus on $d=1$, but the proof in higher dimensions follows the 
same pattern.

First observe that setting $F(x)\equiv x$ 
in~(\ref{neutral bit}), 
\begin{multline}
{\cal L}^{neu}_n\psi_{\mathtt{Id},f}(w_n)
= u_n n^{4/3}\int_{\IR}\overline{w}_n(y)\int_{B_n(y)}\phi_f(x)(1-w_n(x))\mathrm{d}x\mathrm{d}y
\\-u_n n^{4/3}\int_{\IR}(1-\overline{w}_n(y))\int_{B_n(y)}\phi_f(x)w_n(x)\mathrm{d}x\mathrm{d}y
\\=n^{4/3} u_n\int\int \phi_f(x)\ind_{\{|x-y|<R_n\}}\Big(\overline{w}_n(y)-w_n(x)\Big)
\mathrm{d}x\mathrm{d}y.
\end{multline}
Writing $\overline{w}_n$ as an integral and using the notation
$$\Lambda(x)=\max (1-|x|,0), \qquad \Lambda_R(x)=\frac{1}{R}\Lambda\left(
\frac{x}{R}\right),$$
noting that, in $d=1$, $\mathtt{vol}(B_R(x)\cap B_R(y))=(2R)^2\Lambda_{2R}(x-y)$,
we obtain 
\begin{equation*}
{\cal L}^{neu}_n\psi_{\mathtt{Id},f}(w_n)
=
V_Rn^{2/3}\bar{u}\int_{\IR}\phi_f(x)\left(
\int \Lambda_{2R/n^{1/3}}(x-y)\Big(w_n(y)-w_n(x)\Big)\mathrm{d}y\right)\mathrm{d}x.
\end{equation*} 
We now choose a special class of time dependent test functions
$\phi_t^n$, in such a way that 
$${\cal L}^{neu}_n(\langle w_n(s), \phi_{t-s}^n\rangle)=0.$$
In order to make contact with the notation of \cite{mueller/tribe:1995},
we set $\phi_t^n=\psi_t^{n^{2/3}}$ where 
$\psi_t^n$ satisfies
$$\frac{\partial}{\partial t}\psi_t^n(x,z)=
n\bar{u}\left(\int_{\IR}\left(1-\frac{\sqrt{n}|x-y|}{2}
\right)_+
\Big(\psi_t^n(y,z)-\psi_t^n(x,z)\Big)\mathrm{d}y\right),$$
with initial condition $\psi_0^n(x;y) = \delta_{y}(x)$.
Note that this gives an approximation to the Brownian transition 
density as $n\rightarrow\infty$.
Writing $\Lambda_{2/\sqrt{n}}^{*k}$ for the $k$-fold convolution of
$\Lambda_{2/\sqrt{n}}$,
$$\psi_t^n(x,z)=\sum_{k=0}^\infty e^{-nt}\frac{(nt)^k}{k!}
\Lambda_{2/\sqrt{n}}^{*k}(x-z).$$

We shall require the following lemma, which is essentially 
Lemma~3 of \cite{mueller/tribe:1995}, and can be proved in the same way, using
bounds from the local central limit theorem.  
\begin{lemma}
\label{bounds on psi}
Let $\|f\|_\lambda=\sup_x|f(x)|e^{\lambda |x|}$. Then, for all $x$, $y$, $z$,
\begin{enumerate}
\item  \label{MT 1}
$|\psi_t^n(x,z)-e^{-nt}\delta_z(x)|\leq C_T t^{-1/2}$ for $t\leq T$.
\item \label{MT 2}
$|\psi_t^n(x,y)-\psi_t^n(x,z)-e^{-nt}\delta_y(x)+e^{-nt}\delta_z(x)|\leq 
C(t^{-1}|y-z|+n^{-1}t^{-3/2})$.
\item\label{MT 3}
$\psi_t^n(x,z)\leq C_{\lambda,T}e^{-\lambda |x-z|}$ for 
$\lambda>0$, $t\leq T$ and $|x-z|\geq 1$.
\item \label{MT 4}
$\|\psi_t^n(x,z)-e^{-nt}\delta_z(x)\|_\lambda\leq C_{\lambda, T}
t^{-1/2}e^{\lambda |z|}$ for $\lambda>0$, $t\leq T$.
\item\label{MT 5}
For $\lambda >0$, $t\leq T$,
\begin{multline*}
\|\psi_t^n(x,y)-\psi_t^n(x,z)-e^{-nt}\delta_y(x)+e^{-nt}\delta_z(x)\|_\lambda
\\
\leq 
C_{\lambda,T}(t^{-1/2}|y-z|^{1/2}+n^{-1/2}t^{-3/4})e^{\lambda |z|}
\end{multline*}
\end{enumerate}
\end{lemma}
Now using~(\ref{sum of two terms}) with $F(x)=x$ 
and~(\ref{eq 8.23.5})
we find that
\begin{align*}
w_{n}(t,x) = M_{t}(\phi^n_{\cdot}(\cdot,x)) + 
\int_{0}^{t}\int_{\IR}\phi_{t-s}^{n}(z,x)v_{n}(s,z) \mathrm{d}z\mathrm{d}s
+{\mathcal O}(n^{-1/3}),
\end{align*}
where $M_t(\phi^n_{\cdot}(\cdot,x))$ is a martingale and $v_n(s,z)$ is
uniformly bounded over compact time intervals.
We can then write
\begin{multline*}
|w_{n}(t,x) - w_{n}(t,y)|  
\\ 
=
\left| M_{t}(\phi^n_{\cdot}(\cdot,x)) - 
M_{t}(\phi^n_{\cdot}(\cdot,y)) + \int_{0}^{t}\int_{\IR}\Big(\phi^n_{t-s}(z,x) - 
\phi^n_{t-s}(z,y)\Big)v_{n}(s,z)\mathrm{d}z\mathrm{d}s\right|.
\end{multline*}
We first control the integral on the right. 
Using the bounds in Lemma~\ref{bounds on psi} we write 
\begin{align*}
\left|\int_{0}^{t}\int_{\IR}\Big(\phi_{t-s}^n(z,x) - \phi_{t-s}^n(z,y)\Big)
v_{n}(s,z)
\mathrm{d}z\mathrm{d}s \right|  \leq I_{1} + I_{2} + I_{3}
\end{align*}
where
\begin{align*}
I_{1} &= \left|\int_{0}^{t- \tau}\int_{\IR}\Big(\phi^n_{t-s}(z,x)  -\phi^n_{t-s}(z,y)\Big)v_{n}(s,z)\mathrm{d}z\mathrm{d}s  \right.
\\
& \phantom{AAAAAAAAAAAAAAAAAAAaaaaa}\left.
-\int_{0}^{t- \tau}e^{-n(t-s)}\left(v_{n}(s,x) - v_{n}(s,y)\right) 
\mathrm{d}s\right|,
\\
I_{2} &=\left|\int_{t- \tau}^{t}\int_{\IR}\Big(\phi^n_{t-s}(z,x) - 
\phi_{t-s}(z,y)\Big)v_{n}(s,z)\mathrm{d}z\mathrm{d}s \right.  
\\
& \phantom{AAAAAAAAAAAAAAAAAAAaaaaa}\left. - 
\int_{t- \tau}^{t}e^{-n(t-s)}\left(v_{n}(s,x) - v_{n}(s,y)\right)
\mathrm{d}s \right|,\\
I_{3} &= \left|\int_{0}^{t}e^{-n(t-s)}\left(v_{n}(s,x) - 
v_{n}(s,y)\right) \mathrm{d}s\right|.
\end{align*}
We estimate these three terms separately. First observe that
\begin{align*}
I_{1} &= \left| \int_{0}^{t- \tau}\int_{\IR}\left(\phi^n_{t-s}(z,x) - 
\phi^n_{t-s}(z,y) + e^{-nt}\delta_{x}(z) -  e^{-nt}\delta_{y}(z)
\right)v_{n}(s,z)\mathrm{d}z\mathrm{d}s \right| \\
&\leq \sup_{s \in [0,t - \tau]}\| \phi^n_{t-s}(z,x) - \phi^n_{t-s}(z,y) + 
e^{-nt}\delta_{x}(z) -  e^{-nt}\delta_{y}(z) \|_{\lambda}
\\
&\phantom{\leq \sup_{s \in [0,t - \tau]}\| \phi^n_{t-s}(z,x) - \phi^n_{t-s}(z,y) + 
e^{-nt}\delta_{x}}\times\left| 
\int_{0}^{t - \tau}\int_{\IR}e^{-\lambda z}v_{n}(s,z)\mathrm{d}z\mathrm{d}s 
\right| \\
& \leq C_{\lambda,T}\left(\tau^{-\frac{1}{2}}|x - y|^{\frac{1}{2}} + 
n^{-\frac{1}{3}}\tau^{-\frac{3}{4}} \right)e^{\lambda |y|},
\end{align*}
where we have used Part~\ref{MT 5} of Lemma~\ref{bounds on psi}.  
Analogously, 
\begin{align*}
I_{2} &= \left| \int_{t-\tau}^{t}\int_{\IR}\left(\phi^n_{t-s}(z,x) - 
\phi^n_{t-s}(z,y) + e^{-nt}\delta_{x}(z) -  e^{-nt}\delta_{y}(z)
\right)v_{n}(s,z)\mathrm{d}z\mathrm{d}s \right| \\
& \leq C_{\lambda} \int_{t-\tau}^{t}\|\phi^n_{t-s}(z,x) - \phi^n_{t-s}(z,y) 
+ e^{-nt}\delta_{x}(z) -  e^{-nt}\delta_{y}(z)\|_{\lambda}
\mathrm{d}s \\
& \leq C_{\lambda, t}\tau^{\frac12} e^{\lambda |y|},
\end{align*}
where we have used Part~\ref{MT 4} of Lemma~\ref{bounds on psi}. 
Finally, $I(3)$ can be bounded by 
\begin{align*}
I_{3} \leq C \int_{0}^{t}e^{-n^{2/3}s}\mathrm{d}s \leq \frac{C}{n^{2/3}}
\end{align*}
Combining the estimates above we conclude that 
\begin{multline}\label{A1 bound on integral term}
|w_{n}(t,x) - w_{n}(t, y)|
\\
\leq
\left| M_{t}(\phi^n_{\cdot}(\cdot,x)) - 
M_{t}(\phi^n_{\cdot}(\cdot,y))\right| 
+\left|\int_{0}^{t}\int_{\IR}\lbrace
\phi_{t-s}^n(z,x) - 
\phi_{t-s}^n(z,y)
\rbrace
v_{n}(s,z)\mathrm{d}z\mathrm{d}s \right|
\\
\leq 
\left| M_{t}(\phi^n_{\cdot}(\cdot,x)) - 
M_{t}(\phi^n_{\cdot}(\cdot,y))\right| 
+ 
C\left(|x-y|^{1/4}+n^{-1/12}\right)e^{\lambda|x|},
\end{multline}
where the second term is independent of $s\leq t$.
We turn our attention to the difference of martingale terms. 
We use the Burkholder-Davis-Gundy inequality in the form:
\begin{align*}
\IE\left[ \sup_{s \leq t} M_{s}\right] \leq C\left( 
\IE \langle M \rangle_{t}^{\frac{1}{2}} + 
\sup_{s\leq t} |M_{s} - M_{s_{-}}| \right).
\end{align*}
The first term can be controlled in the same way as above, by writing it
as an integral of $\phi_{\cdot}^n(\cdot ,x)-\phi_{\cdot}^n(\cdot ,y)$ times
a bounded function and the jumps of the martingale are of size 
${\mathcal O}(n^{-1/3})$, so combining with~(\ref{A1 bound on integral term})
leads to, for any $\lambda > 0$,
\begin{equation}
\label{ctty bound eq}
\IE[\sup_{s\leq t}|\overline{w}_n(s,x)-\overline{w}_n(s,y)|] \leq C_{\lambda,t}
\left(\left(n^{-\frac{1}{3}} + |x - y|^{\frac{1}{4}} + n^{-\frac{1}{12}}\right)e^{\lambda |x|}
\right).
\end{equation}

Taking into account that $\rho_{\epsilon}$ is a probability density supported in $B_{0}(\epsilon)$ and that the support of the test function $f$ is compact support, we conclude that \eqref{Appendix Fubini} can be bounded by
\begin{multline*}
\IE\left[\left|\int_{\IR^d}f(x)\overline{w}^2_n(s,x)\mathrm{d}x
-\int_{\IR^d}\int_{\IR^d}f(x)\overline{w}_n(s,x)\overline{w}_n(s,y)
\rho_\epsilon(x-y)\mathrm{d}y\mathrm{d}x\right|\right] \\
\leq C_{\lambda,t}
\left(n^{-\frac{1}{3}} + \epsilon^{\frac{1}{4}} + n^{-\frac{1}{12}}\right)
\end{multline*} 
Noting once again that the remaining term can be bounded in the same way, we let $n \to \infty$ to obtain  
\begin{align*}
\limsup_{n\rightarrow\infty}\IE\left[\left|\langle \overline{w}_n^k(s),
f\rangle-\langle (w^\infty)^k(s),f\rangle\right|\right] \leq C\epsilon^{\frac{1}{4}},
\end{align*}
which, since $\epsilon$ is arbitrary,
completes the proof of Proposition~\ref{continuity proposition} .

\section{Equivalence of martingale problem and stochastic p.d.e.~formulations} 
\label{mgpspde}

Suppose that $d\geq 2$.
Write $w$ for a continuous representative of the density of the ${\cal M}_\lambda$-valued
process, which exists as a result of the
calculations of Appendix~\ref{continuity estimates}. Our aim is to show
that $w$ is a weak solution to the stochastic 
p.d.e.~\eqref{Limiting equation nd}.
This follows as a special case from the following result.
\begin{theorem}\label{APB Equivalence theorem}
Let $b, \sigma, g$ be functions satisfying the conditions of 
Theorem~\ref{rippl sturm theorem}.
Suppose that for all $F \in C^{\infty}(\mathbb{R})$ and $f\in C_c^\infty(\IR^d)$,  
\begin{equation}
\label{MGP1}
F(\langle w_t,f\rangle) - F(\langle w_0,f\rangle) 
-\int_0^t
F^{\prime}(\langle w_s,f\rangle)
\left\{\frac{1}{2}\langle w,\Delta f\rangle + \langle b(s,w),f\rangle
\right\}\mathrm{d}s
\end{equation}
is a martingale with quadratic variation
\begin{equation}
\nonumber
\int_0^t\int_{\IR^d\times\IR^d} F^{\prime\prime}(\langle w_s,f\rangle) g(x,y)\sigma(s,x,w))\sigma(s,y,w))f(x)f(y)\mathrm{d}x\mathrm{d}y\mathrm{d}s.\end{equation}
Then $w$ is a (stochastically and analytically) weak solution to the 
stochastic p.d.e.
\begin{equation}
\label{SPDE2}
dw=\left\{\frac{1}{2}\Delta w+ b(t,x,w)\right\}dt+ \sigma(t,x,w)W(\mathrm{d}t,\mathrm{d}x),
\end{equation}
where the noise $W$ is white in time and coloured in space with 
quadratic variation
$$\langle W(\phi)\rangle_t=t\int_{\IR^d\times \IR^d}\phi(x)\phi(y)g(x,y)\mathrm{d}x\mathrm{d}y.$$
\end{theorem}
The converse of 
Theorem~\ref{APB Equivalence theorem} is a simple application of  
It\^o's formula.  

We closely follow  \cite{kurtz:2010}, who 
provides a powerful
proof of the corresponding result for stochastic 
(ordinary) differential equations, not through the classical 
approach of explicitly constructing the driving noise, but via
the Markov Mapping Theorem. The proof of our case requires only very 
minor modifications of the proof from  \cite{kurtz:2010}.
This proof is very flexible and could be adapted further without any 
major difficulties to cover a wider class of equations, e.g.~with
more singular coefficients $b, \sigma$ in \eqref{SPDE2}.
For convenience, we first recall the Markov Mapping Theorem in the form stated there.

\begin{theorem}[Markov Mapping Theorem,  \cite{kurtz:2010}, Theorem 1.4] \label{APB MMT}
Suppose that $E$ is a complete separable metric space and that
the operator
$B\subseteq \overline{C}(E)\times\overline{C}(E)$ is 
separable and a pre-generator and 
that its domain ${\cal D}(B)$ is closed under
multiplication and separates points in $E$. 
Let $(E_0,r_0)$ be a complete, separable, metric space,
$\gamma:E\rightarrow E_0$ be Borel measurable, and $\alpha$ be a 
transition function from $E_0$ into $E$ 
($y\in E_0\rightarrow \alpha(y,\cdot)\in {\cal P}(E)$ is Borel measurable)
satisfying $\alpha(y,\gamma^{-1}(y))=1$.
Define  
$$C=\left\{\left(\int_Ef(z)\alpha(\cdot,dz),\int_EBf(z)\alpha(\cdot,dz)\right)
:f\in {\cal D}(B)\right\}.$$
Let $\mu_0\in {\cal P}(E_0)$, and define $\nu_0=\int\alpha(y,\cdot)\mu_0(dy)$.
If $\widetilde{U}$ is a solution of the martingale problem for 
$(C,\mu_0)$, then there exists a solution $V$ of the martingale problem
for $(B,\nu_0)$ such that $\widetilde{U}$ has the same distribution on
$M_{E_0}[0,\infty)$ (the space of measurable functions mapping 
$[0,\infty)$ to $E_{0}$, with topology given by convergence in 
Lebesgue measure)
as $U=\gamma\circ V$ and 
\begin{equation}
\IP\left[V(t)\in \Gamma |\widehat{\cal F}_t^U\right]=
\alpha(U(t),\Gamma), \quad \Gamma\in{\cal B}(E), t\in T^U,
\end{equation}
where $\widehat{\cal F}_t^U$ is the completion of the $\sigma$-field $\sigma\left( \int_{0}^{r}h(U(s))\mathrm{d}s : r \leq t \right)$ (for bounded and measurable $h$) and $T^{U}$ is the set of times for which $U$ is measurable.
\end{theorem} 

The proof of Theorem~\ref{APB Equivalence theorem}
is based on the introduction of an auxiliary  process $Z$, whose 
generator plays the role of $B$ in the statement of Theorem~\ref{APB MMT}. 
This process is introduced in such a way that we can evaluate the 
conditional distributions of the driving noise given the state of the 
process $w$. Those conditional distributions are given in terms of 
stationary processes.

First we describe the representation of the driving noise 
in terms of a countable collection of stationary processes.
By our assumptions, the noise can be realised as an element of 
the Hilbert space $W^{-1 - d,2}(\mathbb{R}^{d})$ (see, 
e.g.~\cite{rippl:2012}), that is on the dual of 
the Sobolev space $W^{1+d,2}(\mathbb{R}^{d})$. Therefore there exists 
an orthonormal basis  $\{\phi_i\}_{i\geq 1}$ 
such that the noise $W$ is  completely characterised by 
a countable sequence of independent, one-dimensional Brownian motions
$W_t(\phi_{i})$ with  
$$ W_t(\phi_i)=\int_0^t\int_{\IR^d}\phi_i(x)W(ds,dx),$$
and, for any adapted process $H$, the integral with respect to the noise 
can be written in the form
$$\int_0^t\int_{\IR^d}H(s-,x)W(ds,dx)=\sum_{i=1}^\infty\int_0^t
\langle H,\phi_i\rangle\phi_i\mathrm{d}W_s(\phi_i).$$

Define $Y=\{Y_i\}_{i\geq 1}$ by 
$$Y_i=\bigg(Y_i(0)+ W_t(\phi_i)\mod 2\pi\bigg).$$
We can recover $W_t(\phi_i)$ (and hence $W$) from $Y_i$, by observing the 
increments. Indeed, setting
\begin{equation*}
\beta_i(t)=\binom{\cos(Y_i(t))+\frac{1}{2}\int_0^t\cos(Y_i(s))ds}{
\sin(Y_i(t))+\frac{1}{2}\int_0^t\sin(Y_i(s))ds},
\end{equation*}
we have
\begin{equation}\label{APB recovery of W}
W_t(\phi_i)=\int_0^t\left(-\sin(Y_i(s)), \cos(Y_i(s))\right)d\beta_i(s)
\end{equation}
and
$$d\beta_i(t)=\binom{-\sin(Y_i(t))}{\cos (Y_i(t))}dW_t(\phi).$$

\begin{remark}\label{APB properties of Y}
Our definition ensures that $Y_i(t)$ are independent.
Identifying $0$ and $2\pi$, we note that $Y=\{Y_i\}_{i\geq 1}$ is a 
Markov process with compact state space $[0,2\pi)^\infty$. 
If $Y_{i}(0)$ is uniformly distributed on $[0,2\pi)$ and independent of $W$, 
then $Y_{i}$ is  stationary.
\end{remark}

We define an auxiliary process $Z$ by
$$Z=\binom{w}{Y},$$
where $w$ is a solution of the martingale problem~(\ref{MGP1})
and $Y = \{Y_{i}\}_{i\geq 1}$.
Let $A$ be the generator associated with the 
martingale problem~(\ref{MGP1}).
Its domain is given by 
\begin{align*}
\mathcal{D}(A) = \bigg\{    
F(\langle \cdot, f \rangle) : f \in C_{c}^{\infty}, F \in C^{\infty}(\mathbb{R})
\bigg\}.
\end{align*}
Let $D_0([0,2\pi))$ be the collection of functions for which 
$$f(0)=f(2\pi_{-}), \quad f^{\prime}(0)=f^{\prime}(2\pi_{-}),
\quad f^{\prime \prime}(0)=f^{\prime \prime}(2\pi_{-}).$$
Let $\widehat{A}$ be the generator
of the process $Z$.
Its domain $\mathcal{D}(\widehat{A})$ is defined by
$$D(\widehat{A})=\left\{\psi(w)\prod_{i=1}^m f_i(y_i):
\psi\in \mathcal{D}(\mathcal{A}), f_i\in D_0\right\}.$$
An application of It\^o's formula guarantees that $\widehat{A}$ can be
written in the form
$$
\widehat{A}\big[
\psi(w)
\prod_{i}^{m}f_{i}(y)
\big]
= 
\prod_{i}^{m}f_{i}(y)A\psi (w)
+
\psi(w)
\sum_{i}^{m} {\cal L}_if_i(y_{i})\prod_j^if_{j}(y_j)
+
\sum_{i}c_{i}\partial_{y_{i}}f_{i},
$$
where ${\cal L}_i$ is the generator of $Y_i$, $\prod_j^i$ denotes the 
product over $j=1,\ldots,m$ with the $i$th term omitted and 
the last term (which is finite due to our construction of the noise) 
is determined by the Meyer process between $w$ and $Y_i$. 
The exact value of the $c_i$'s
will not concern us. We write it in this form to stress 
the dependence on the derivative of $f_{i}$.

To show that solutions to the martingale problem for $\widehat{A}$  
are weak solutions to~\eqref{SPDE2}, we use the following lemma.
\begin{lemma}[ \cite{kurtz:2010}, Lemma~A.1]
\label{APB Maritngale Lemma}
Let $A \subset B(E) \times B(E)$ be a generator, and let $X$ be a 
c\`adl\`ag solution of the martingale problem for $A$. 
For each $f \in \mathcal{D}(A)$, define
\begin{align*}
N_{f} = f(X_{t}) - \int_{0}^{t} Af(X_{s})\mathrm{d}s.
\end{align*}
Suppose $\mathcal{D}(A)$ is an algebra and that $f\circ X$ is c\`adl\`ag 
for each $f \in \mathcal{D}(A)$.
Let $f_{0} \cdots, f_{m} \in \mathcal{D}(A)$ and 
$h_{0}, \cdots, h_{m} \in B(E)$. Then
\begin{align*}
N(t) = \sum_{i} \int_{0}^{t} h(X_{s_{-}})\mathrm{d}N_{f_{i}}(s)
\end{align*}
is a square integrable martingale with Meyer process
\begin{multline}
\langle N \rangle_{t} = \sum_{i,j}\int_{0}^{t}h_{i}(X_{s})h_{j}(X(s))
\\
\times
\bigg[Af_{i}f_{j}(X_s) - f_{i}(X_s)Af_{j}(X_s) - 
f_{j}(X_s)Af_{i}(X_s)\bigg] \mathrm{d}s
\end{multline}
\end{lemma}
\begin{lemma}
Every solution of the martingale problem defined by $\widehat{A}$ is 
a solution 
of~\eqref{SPDE2}, with $W$ defined by~\eqref{APB recovery of W}.
\end{lemma}
\begin{proof}
The statement follows from an application of 
Lemma~\ref{APB Maritngale Lemma} 
with 
\begin{align*}
f_{0}(w,\bar{y})&= F(\langle w,f\rangle),
\quad
f_{i}(w,\bar{y}) = \cos(y_{i}) \text{ for } i \in \{ 1, \dots m\}, 
\\
f_{i}(w,\bar{y}) &= \sin(y_{i}) \text{ for } i \in \{ m+1, \dots 2m\}, \\
h_{0}(w,\bar{y}) &= 1, \quad h_{i}(w,\bar{y}) = c_{i}\sin(y_{i}) \text{ for } i \in \{ 1, \dots m\}, 
\\ 
h_{i}(w,\bar{y}) &= -c_{i}\cos(y_{i}) \text{ for } i \in \{ m+1, \dots 2m\},
\end{align*}
where we have used $\bar{y}$ to denote the vector $(y_i)_{i\geq 1}$.
This choice of functions guarantees that 
$N(t)$ defined as in Lemma~\ref{APB Maritngale Lemma} is a 
martingale with quadratic variation equal to $0$. 
This, in turn, implies that $w$ is a weak solution to \eqref{SPDE2}.
\end{proof}

\begin{proof}[Proof of Theorem \ref{APB Equivalence theorem}]
We begin by observing that $D(\widehat{A})$ is separable and closed under
multiplication. The next step is to construct the transition function $\alpha$
in the Markov Mapping Theorem. 
Let ${\cal P}(\Omega)$ be the collection of probability measures on $\Omega$. 
Let $\eta_y\in {\cal P}([0,2\pi)^{\infty})$ be the product of independent 
uniform distributions on $[0,2\pi)$. We define  
$$\alpha (w,\cdot)=\delta_{w}\times \eta_y\in {\cal P}
\left(C_{b}(\mathbb{R}^{d})\times [0,2\pi)^{\infty}\right).$$
We observe that 
\begin{equation}
\label{observation 1}
\alpha\bigg( A[F(\langle w, f \rangle)
\prod_{i}^{m}f_{i}(y)] \bigg)= A\alpha\bigg( F(\langle w, f \rangle)
\prod_{i}^{m}f_{i}(y) \bigg).
\end{equation}
Since $\eta_y$ has been chosen in a way which guarantees that it is a 
stationary distribution for $Y$, we have, by definition,
\begin{equation}
\label{observation 2}
\alpha \bigg( F(\langle w, f \rangle)
\sum_{i}^{m} {\cal L}_i f_i(y_{i})\prod^i_{j}f_{j}(y) \bigg) = 0.
\end{equation}
Finally, since the functions $f_{i}$ are chosen to be periodic, 
\begin{equation}
\label{observation 3}
\alpha \bigg( \sum_{i}c_{i}\partial_{y_{i}}f_{i}\bigg) = 0.
\end{equation}
Combining~\eqref{observation 1},~\eqref{observation 2},~\eqref{observation 3}, we conclude that
$$\alpha \bigg( \widehat{A}F(\langle w, f \rangle)
\prod_{i}^{m}f_{i}(y) \bigg) =A\alpha \bigg(F(\langle w, f \rangle)
\prod_{i}^{m}f_{i}(y) \bigg).$$
Our result now  follows from application of the Markov  Mapping Theorem~\ref{APB MMT} .
\end{proof}

\section{Detailed description of simulations}
\label{code}
\subsection{A general description}

We provide a more detailed description of the simulation. 

\paragraph{Initialization}
The simulation begins by choosing random seeds, separately for the times 
of reproduction/migration events and for the selection of individuals during 
those events.
The user specifies the parameters of the simulation: the number of demes, their spatial structure 
(for the simulations reported in Section~\ref{numerics} this 
is always a one-dimensional torus), the number of individuals per deme, 
a correlation function for the environments, the selection coefficient, 
the rate of changes in the environment, the frequency of creating records 
and the time of the simulation. 
The rates of neutral and 
migration rates are determined by the number of individuals per deme, and the
rate of selection events depends additionally on the selection coefficient.

The basic objects in the simulation are initialized. Each deme is 
assigned an identification number and a list of neighbours, 
determined by the selected spatial structure (for the one-dimensional 
torus, each deme has two neighbours). Each individual is randomly 
assigned a type according to a Bernoulli (1/2) distribution. 
Each individual is assigned an identification number, according to
the deme in which it lives.
This number allow us to track the ancestral origin of the 
population forward in time. The initial state of the environment in
each deme is specified. A maximal time of the simulation is set.

The object that determines the times of the events, `the Clock', 
is initialised. For each deme, for each type of event, the next time 
at which an event occurs is encoded. A separate variable with the time of 
the next environment change is also kept. The Clock is initialised with 
the exponential distribution with rates specified for each type of an event. 
The current time is set to $0$.

\paragraph{Simulation} 
As long as the time of the simulation is smaller than the maximal 
time of the simulation the following is repeated.

The clock is asked to return the time, deme identification number $\alpha$, 
and type $\kappa$ of the next event (chosen to be the smallest number on 
the list of all times stored). This time is saved as the current time in 
the simulation. A random number is drawn from an exponential distribution 
with rate determined by the type $\kappa$ of the event. 
This new time is added to the current 
time and passed to the Clock to specify the next time an event of type
$\kappa$ occurs 
in deme $\alpha$.

The event of type $\kappa$ is executed in deme with identification 
number $\alpha$. With a given frequency (specified by the user
with respect to the number of events that have taken place), 
a record of the current state of the population is created.  

\paragraph{Events}
There are four type of events: two types of reproduction events 
(neutral and selective), migration events and environmental events. 

Whenever a reproduction event takes place in deme $\alpha$, two individuals 
are chosen uniformly from deme $\alpha$. One of the selected individuals 
is then randomly chosen to be a potential parent, using a 
Bernoulli (1/2) distribution. If the event is neutral, the potential parent 
is selected as the parent. If the event is selective, the types of the 
two selected individuals are checked. If they are the same, the 
potential parent is selected as the parent. If the types are different, 
the state of the environment is checked. Then the individual with type 
favoured by the environment is selected as the parent. 
The type and the ancestral origin of the parent are assigned to the 
second individual.

Whenever a migration event occurs in deme $\alpha$, one of the neighbouring 
demes is randomly selected according to the uniform distribution. A single 
individual is chosen uniformly from each of deme $\alpha$ and the randomly
selected neighbour and the type and ancestral origin of those individuals 
are swapped. 

Whenever an environmental event occurs, a new state of the environment 
is randomly chosen. For the environmental regimes considered 
in Section~\ref{numerics}, a single number ($1$ or $-1$) is drawn randomly 
from a Bernoulli $(1/2)$ distribution. In the cases of neutral/constant selection,
the environment is not changed. In the fluctuating cases, the environment is 
changed in all demes (consistent with the specified correlation function).

\paragraph{Recording}
Two separate types of record are created, the `proportion file' and 
the collection of `ancestral origin files'. The number of ancestral 
origin files is equal to the number of demes, $\beta$.

The proportion file is a .txt file with $\beta + 3$ columns, separated by a 
single space. The first column contains the information of the time at which 
records are created.
The second column contains the global proportion of type $a$ at the
given time. Columns $3, \ldots ,(\beta +3)$ record 
the proportion of type $a$ in 
each deme. 

Each `ancestral origin' file is a .txt file with $\beta + 1$ columns, 
separated by a single space. File number $\gamma$ contains information 
on the proportion of individuals in each deme with ancestral origin equal 
to $\gamma$. The first column contains the time of creation of the records. 
Columns $2,\ldots ,(\beta +1)$ record the proportion of individuals with 
ancestral origin equal to $\gamma$ in each deme.

\subsection{Implementation}
The simulation was implemented in C\texttt{++}14. The Mersenne Twister 
pseudorandom number generator was used. Figures were created using R, 
(2015, \cite{R_Manual:2015}).

\end{appendix}

\addcontentsline{toc}{section}{References}
\DeclareRobustCommand{\VAN}[3]{#3}

\end{document}